\documentclass[12pt]{article}

\usepackage[margin=1in]{geometry}

\usepackage{amsmath}
\usepackage{amssymb}
\usepackage{amsthm}
\usepackage{amsfonts}
\usepackage{enumerate}
\usepackage{graphicx}
\usepackage{url}
\usepackage[table]{xcolor}
\usepackage{soul}
\usepackage{subcaption}
\usepackage{tikz}
\usetikzlibrary{arrows,shapes,positioning,decorations.markings,arrows.meta}
\tikzset{>=stealth}
\tikzset{->-/.style={decoration={
  markings,
  mark=at position .5 with {\arrow[scale=1.5]{>}}},postaction={decorate}}}
\tikzset{-<-/.style={decoration={
  markings,
  mark=at position .5 with {\arrow[scale=1.5]{<}}},postaction={decorate}}}
\tikzset{->--/.style={decoration={
  markings,
  mark=at position .3 with {\arrow[scale=1.5]{<}}},postaction={decorate}}}
\tikzset{-<--/.style={decoration={
  markings,
  mark=at position .3 with {\arrow[scale=1.5]{>}}},postaction={decorate}}}
\tikzset{-->-/.style={decoration={
  markings,
  mark=at position .7 with {\arrow[scale=1.5]{<}}},postaction={decorate}}}
\tikzset{--<-/.style={decoration={
  markings,
  mark=at position .7 with {\arrow[scale=1.5]{>}}},postaction={decorate}}}

\usepackage{braids}

\usepackage{array}

\newtheorem{theorem}{Theorem}[section]

\newtheorem{lemma}[theorem]{Lemma}
\newtheorem{example}[theorem]{Example}
\newtheorem{corollary}[theorem]{Corollary}
\newtheorem{proposition}[theorem]{Proposition}

\def\Z{\mathbb{Z}}

\date{}
\begin{document}

\title{Oriented Local Moves and\\Divisibility of the Jones Polynomial}

\author{
Paul Drube\\[-2ex]
\footnotesize Department of Mathematics \& Statistics\\[-2.5ex]
\footnotesize Valparaiso University\\[-2.5ex] 
\footnotesize\tt paul.drube@valpo.edu
\and
Puttipong Pongtanapaisan\\[-2ex]
\footnotesize Department of Mathematics\\[-2.5ex]
\footnotesize University of Iowa\\[-2ex]
\footnotesize\tt puttipong-pongtanapaisan@uiowa.edu
}

\maketitle

\begin{abstract}
For any virtual link $L = S \cup T$ that may be decomposed into a pair of oriented $n$-tangles $S$ and $T$, an oriented local move of type $T \mapsto T'$ is a replacement of $T$ with the $n$-tangle $T'$ in a way that preserves the orientation of $L$.  After developing a general decomposition for the Jones polynomial of the virtual link $L = S \cup T$ in terms of various (modified) closures of $T$, we analyze the Jones polynomials of virtual links $L_1,L_2$ that differ via a local move of type $T \mapsto T'$.  Succinct divisibility conditions on $V(L_1)-V(L_2)$ are derived for broad classes of local moves that include the $\Delta$-move and the double-$\Delta$-move as special cases.  As a consequence of our divisibility result for the double-$\Delta$-move, we introduce a necessary condition for any pair of classical knots to be $S$-equivalent.
\end{abstract}

\section{Introduction}
For any link $L$, the Jones polynomial $V(L) \in \Z[t^{1/2},t^{-1/2}]$ is a Laurent polynomial in the variable $t^{1/2}$.  After being introduced by Jones himself \cite{Jones}, the Jones polynomial was recast by Kauffman in terms of his bracket polynomial \cite{Kauffman1} .  For any unoriented link diagram $L$, the bracket polynomial $\langle L \rangle \in \Z[A,A^{-1}]$ is an invariant of framed links that may be defined recursively via the local relations shown below.

\begin{center}
\raisebox{9pt}{\scalebox{2}{$\langle$}}
\raisebox{6pt}{\scalebox{.45}{\begin{tikzpicture}
[scale=1.5,auto=left,every node/.style={circle,inner sep=0pt}]
\draw[line width=1.4pt] (-.5,-.5) to (.5,.5) {};
\draw[line width=1.4pt] (.5,-.5) to (.1,-.1) {};
\draw[line width=1.4pt] (-.1,.1) to (-.5,.5) {};
\end{tikzpicture}}}
\raisebox{9pt}{\scalebox{2}{$\rangle$} $\ = \ A$ \scalebox{2}{$\langle$}}
\raisebox{6pt}{\scalebox{.45}{\begin{tikzpicture}
[scale=1.5,auto=left,every node/.style={circle,inner sep=0pt}]
\draw[bend left=80,line width=1.4pt] (.5,-.5) to (.5,.5) {};
\draw[bend right=80,line width=1.4pt] (-.5,-.5) to (-.5,.5) {};
\end{tikzpicture}}}
\raisebox{9pt}{\scalebox{2}{$\rangle$} $\ + \ A^{-1}$ \scalebox{2}{$\langle$}}
\raisebox{6pt}{\scalebox{.45}{\begin{tikzpicture}
[scale=1.5,auto=left,every node/.style={circle,inner sep=0pt}]
\draw[bend left=80,line width=1.4pt] (-.5,-.5) to (.5,-.5) {};
\draw[bend left=80,line width=1.4pt] (.5,.5) to (-.5,.5) {};
\end{tikzpicture}}}
\raisebox{9pt}{\scalebox{2}{$\rangle$}}

\vspace{.1in}

\raisebox{9pt}{\scalebox{2}{$\langle$}}
\raisebox{4pt}{\begin{tikzpicture}
[scale=1,auto=left,every node/.style={circle,inner sep=0pt}]
\node[draw,line width=.8pt,inner sep=6.5pt] (v) at (0,0) {};
\end{tikzpicture}}
\raisebox{9pt}{$ \cup \ L$ \scalebox{2}{$\rangle$} $\ = \ (-A^2-A^{-2})$ \scalebox{2}{$\langle$}} \kern-3pt
\raisebox{9pt}{{$L$}} \kern-3pt
\raisebox{9pt}{\scalebox{2}{$\rangle$}} 

\vspace{.1in}

\raisebox{9pt}{\scalebox{2}{$\langle$}}
\raisebox{4pt}{\begin{tikzpicture}
[scale=1,auto=left,every node/.style={circle,inner sep=0pt}]
\node[draw,line width=.8pt,inner sep=6.5pt] (v) at (0,0) {};
\end{tikzpicture}}
\raisebox{9pt}{\scalebox{2}{$\rangle$} $\ = \ 1$ }
\end{center}

For an oriented link diagram with the bracket polynomial $\langle L \rangle$, one may obtain the Jones polynomial of the associated link by evaluating $f(L) = (-A^3)^{-w(L)} \langle L \rangle$ at $A = t^{-1/4}$, where $w(L)$ is the writhe of $L$.  We henceforth refer to the intermediate polynomial $f(L)$ as the auxiliary polynomial of $L$.



The rest of this paper assumes a basic familiarity with the Jones polynomial and the Kauffman bracket.  For more information on these topics, see Kauffman \cite{Kauffman1} or Lickorish \cite{Lickorish}.

The Jones polynomial was subsequently generalized to virtual links by Kauffman \cite{Kauffman2}.  The resulting virtual link invariant, sometimes referred to as the Jones-Kauffman polynomial, may be defined in terms of the Kauffman bracket using the same local relations as above and the same evaluation of $f(L) = (-A^3)^{-w(L)} \langle L \rangle$ at $A = t^{-1/4}$.  For a full discussion of virtual links and their topological importance, consult the surveys \cite{Kauffman2,Kauffman3}. 

Now consider the unoriented virtual link diagram $D$, and suppose that $D = S \cup T$ may be decomposed into the pair of $n$-tangles $S$ and $T$.\footnote{Throughout this paper, we use a generalized notion of tangle that allows for closed components without endpoints on the boundary.}  An (unoriented) local move of type $T \mapsto T'$ is a replacement of $T$ with the $n$-tangle $T'$ while leaving $S$ unchanged, transforming $D$ into a diagram $D' = S \cup T'$ of some (possibly distinct) virtual link.  Local moves include operations as ubiquitous as the simple crossing change (on $2$-tangles), the $\Delta$-move (on $3$-tangles), and the so-called forbidden moves of virtual links (on $3$-tangles).  An oriented local move of type $T \mapsto T'$ is a replacement of the oriented $n$-tangle $T$ with the oriented $n$-tangle $T'$ in a way that preserves the orientation of all endpoints of $T$.

The primary goal of this paper is to investigate how the auxiliary polynomial of a virtual link behaves under a variety of oriented local moves.  In particular, we consider any pair of oriented links $L_1,L_2$ that differ via a finite sequence of some fixed move, and develop divisibility conditions for the auxiliary polynomial $f(L_1)-f(L_2)$.  This places a necessary condition upon whether a given pair of links may be connected via repeated application of a particular local move and, in the case where $L_1$ is a knot and $L_2$ is unknot, may be used to show that the move in question is not an unknotting move.

Divisibility conditions of the type above date back to Jones \cite{Jones}, who showed $f(K_1)-f(K_2)$ is divisible by $A^{16} - A^{12} - A^4 + 1$ for any pair of classical knots $K_1,K_2$.  Our methods more closely follow that of Ganzell \cite{Ganzell}, who used the bracket polynomial to find divisibility conditions for $f(K_1) - f(K_2)$ when $K_1,K_2$ were a pair of knots connected by various (unoriented) local moves.  Ganzell showed that $f(K_1) - f(K_2)$ is divisible by $A^{12}-1$ for any pair of classical knots that differ by a crossing change, that $f(K_1) - f(K_2)$ is divisible by $A^{16}-A^{12}-A^{4}+1$ for any pair of classical knots that differ by a $\Delta$-move, and that $f(K_1) - f(K_2)$ is divisible by $A^{10}-A^6-A^4+1$ for any pair of virtual knots that differ by a forbidden move.  For additional results of a similar type see Nikkuni \cite{Nikkuni}, who showed that $f(L_1) - f(L_2)$ is divisible by $(A^{-4}-1)^n(A^{-8}+A^{-4}+1)(A^{-8}+1)$ for any pair of oriented classical links $L_1,L_2$ that differ by a $C_n$ move (for every $n \geq 3$).

Now fix the local move $T \mapsto T'$, and for some collection of links $\mathcal{S}$ consider all pairs $L_1,L_2 \in \mathcal{S}$ that are related via a finite sequence of moves of fixed type $T \mapsto T'$.  We say that $p(A) \in \Z[A,A^{-1}]$ is a maximal divisor for $\mathcal{S}$ with respect to $T \mapsto T'$ if, whenever $q(A) \in \Z[A,A^{-1}]$ divides every polynomial of the form $f(L_1)-f(L_2)$, then $q(A)$ divides $p(A)$.  Note that one may immediately conclude that $p(A)$ is a maximal divisor if there exist $L_1,L_2 \in \mathcal{S}$ such that $f(L_1) - f(L_2) = p(A)$.  Such links have been found for every divisor mentioned in the previous paragraph, proving their maximality within the stated collection of links \cite{Jones,Ganzell,Nikkuni}.

Our results differ from those of Ganzell \cite{Ganzell} in that all of our local moves are oriented.  This narrows the classes of links that may be connected via repeated application of a given move, and our divisibility conditions for $f(L_2)-f(L_1)$ need not extend to any pair of links that differ via an unoriented version of the same move.  On the other hand, dealing with oriented moves allows us to more easily tackle local moves with a large number of outgoing strands.  Observe that a (maximal) divisor for some unoriented local move $T \mapsto T'$ may be obtained by separately determining a (maximal) divisor for every orienation that is compatible with both $T$ and $T'$, and then taking the greatest common divisor of those polynomials.

\subsection{Outline}

This paper is organized as follows.  In Section \ref{sec: decompositions of the Jones polynomial} we introduce our general technique for decomposing the auxiliary polynomial of an arbitrary virtual link of the form $T \cup T'$.  Theorem \ref{thm: auxiliary polynomial decomposition} gives $f(T \cup T') = \sum_{m \in \mathcal{P}_n} q_m f(\widetilde{T}^B(m))$, where the $\widetilde{T}^B(m)$ represent the various closures of $T$ (via every 2-equal matching $m$ in $\mathcal{P}_n$) and $q_m \in \Z[A,A^{-1}]$ are unspecified Laurent polynomials that depend upon the structure of $T'$.

In Section \ref{sec: local moves}, we apply Theorem \ref{thm: auxiliary polynomial decomposition} to find maximal divisors for a variety of oriented local moves.  Subsection \ref{subsec: rotational local moves} focuses upon local moves that involve rotation of a classical $n$-tangle by a fixed number of strands.  Subsections \ref{subsec: double-delta} and \ref{subsec: virtual rotation} present lengthier treatments for a pair of local moves that do not conform to the methods of Subsection \ref{subsec: rotational local moves}, namely the double-$\Delta$-move for classical $6$-tangles and a rotational move for virtual $2$-tangles.  In the case of the double-$\Delta$-move this prompts an intriguing new result on S-equivalence of knots, with Corollary \ref{thm: S-equivalence divisibility} stating that two classical knots $K_1,K_2$ may be S-equivalent only if $f(K_1) - f(K_2)$ is divisible by $A^{36} - A^{32} + A^{28} - A^{24} - A^{12} + A^8 - A^4 + 1$.
 
\section{Decompositions of the Jones Polynomial}
\label{sec: decompositions of the Jones polynomial}

For any $n \geq 1$, consider the set $[2n]= \lbrace 1,2,\hdots,2n \rbrace$.  A 2-equal partition of $[2n]$ is a partition of $[2n]$ into $n$ disjoint sets of size $2$.  Every 2-equal partition $P$ may be associated with a (2-equal) matching on the circle, in which the element $i \in [2n]$ corresponds to the point along the unit circle with radial coordinate $\theta = -\frac{\pi i}{n}$, and the points corresponding to $i$ and $j$ are connected via an arc within the unit circle if and only if $i$ and $j$ belong to the same block of $P$.  We denote the set of all such matchings on $2n$ points by $\mathcal{P}_n$.

An element of $\mathcal{P}_n$ is said to be noncrossing if it may be drawn so that no two arcs intersect.  We denote the set of all noncrossing (2-equal) matchings on $2n$ points by $\mathcal{M}_n$.  It is well known that $\vert \mathcal{M}_n \vert = \frac{1}{n+1} \binom{2n}{n}$, the $n^{th}$ Catalan number.  We henceforth refer to any matching via the blocks of the associated partition.  See Figure \ref{fig: basic matchings example} for basic examples.

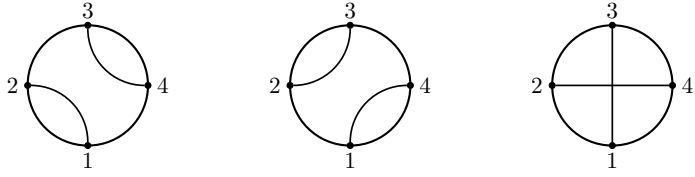
\begin{figure}[ht!]
\centering
\scalebox{.8}{
\begin{tikzpicture}
[scale=1,auto=left,every node/.style={circle,inner sep=0pt}]
	\draw[line width=1pt] (0,0) circle (1cm);
	\node[draw,fill,inner sep=1pt] (1) at (-90:1) {};
	\node[draw,fill,inner sep=1pt] (2) at (-180:1) {};
	\node[draw,fill,inner sep=1pt] (3) at (-270:1) {};
	\node[draw,fill,inner sep=1pt] (4) at (-360:1) {};
	\node (1l) at (-90:1.25) {\small{1}};
	\node (2l) at (-180:1.25) {\small{2}};
	\node (3l) at (-270:1.25) {\small{3}};
	\node (4l) at (-360:1.25) {\small{4}};	
	\draw[thick, bend right=45] (1) to (2);
	\draw[thick, bend right=45] (3) to (4);
\end{tikzpicture}
\hspace{.5in}
\begin{tikzpicture}
[scale=1,auto=left,every node/.style={circle,inner sep=0pt}]
	\draw[line width=1pt] (0,0) circle (1cm);
	\node[draw,fill,inner sep=1pt] (1) at (-90:1) {};
	\node[draw,fill,inner sep=1pt] (2) at (-180:1) {};
	\node[draw,fill,inner sep=1pt] (3) at (-270:1) {};
	\node[draw,fill,inner sep=1pt] (4) at (-0:1) {};
	\node (1l) at (-90:1.25) {\small{1}};
	\node (2l) at (-180:1.25) {\small{2}};
	\node (3l) at (-270:1.25) {\small{3}};
	\node (4l) at (-360:1.25) {\small{4}};	
	\draw[thick, bend left=45] (1) to (4);
	\draw[thick, bend right=45] (2) to (3);
\end{tikzpicture}
\hspace{.5in}
\begin{tikzpicture}
[scale=1,auto=left,every node/.style={circle,inner sep=0pt}]
	\draw[line width=1pt] (0,0) circle (1cm);
	\node[draw,fill,inner sep=1pt] (1) at (-90:1) {};
	\node[draw,fill,inner sep=1pt] (2) at (-180:1) {};
	\node[draw,fill,inner sep=1pt] (3) at (-270:1) {};
	\node[draw,fill,inner sep=1pt] (4) at (-360:1) {};
	\node (1l) at (-90:1.25) {\small{1}};
	\node (2l) at (-180:1.25) {\small{2}};
	\node (3l) at (-270:1.25) {\small{3}};
	\node (4l) at (-360:1.25) {\small{4}};	
	\draw[thick] (1) to (3);
	\draw[thick] (2) to (4);
\end{tikzpicture}}
\caption{The three elements $m_1 = ((1,2),(3,4))$, $m_2 = ((1,4),(2,3))$, $m_3 = ((1,3),(2,4))$ of $\mathcal{P}_2$, among which $m_1$ and $m_2$ also belong to $\mathcal{M}_2$}
\label{fig: basic matchings example}
\end{figure}

In all that follows, we assume that matchings have been drawn such that no two arcs intersect more than once and no three arcs have a common intersection.  Given these conditions, there exists an obvious bijection between $\mathcal{P}_n$ and the set of unoriented virtual $n$-tangles with zero classical crossings.  For any $m \in \mathcal{P}_n$, a diagram of the associated tangle $T_m$ may be obtained by replacing all intersections in $m$ with virtual crossings and interpreting the unit circle as the tangle boundary.  When referring to $T_m$, we will always take a diagram in which the endpoint corresponding to $i$ has radial coordinate $\theta = -\frac{\pi i}{n}$.

Now take any virtual $n$-tangle $T$.  Our formalism involving 2-equal matchings is motivated by the fact that every Kauffman state of $T$ is isotopic to $T_m$ for some $m \in \mathcal{P}_n$.  Gathering terms from the Kauffman state sum that resolve to the same $T_m$, this implies that $\langle T \rangle$ decomposes as $\langle T \rangle = \sum_{m \in \mathcal{P}_n} p_m \langle T_m \rangle$, where each $p_m \in \Z[A,A^{-1}]$ is a Laurent polynomial that depends upon the structure of $T$.  If our tangle $T$ lacks virtual crossings, this decomposition clearly reduces to $\langle T \rangle = \sum_{m \in \mathcal{M}_n} p_m \langle T_m \rangle$.

Directly pertinent to this paper is the situation where a virtual link $L$ may be decomposed into the two $n$-tangles $T$ and $T'$.  In this case, we always take a diagram of $L$ in which $T$ appears as described above, and then order the endpoints of $T'$ so that the $i^{th}$ endpoint of $T'$ is identified with the $i^{th}$ endpoint of $T$.  Here we adopt the shorthand $L = T \cup T'$.

For $L = T \cup T'$, observe that smoothing every crossing in $T'$ (while leaving $T$ unchanged) produces a virtual link $T(m) = T \cup T'_m$ for some $m \in \mathcal{P}_n$.  We refer to this link as the closure of $T$ by $m$.  Diagrammatically, note that $T(m)$ may be obtained from $T$ by inverting all arcs of $m$ across the unit circle, replacing all intersections in the resulting matching with virtual crossings, and attaching the $i^{th}$ endpoint of $m$ to the $i^{th}$ external strand of $T$.

See Figure \ref{fig: 4-strand closures} for an illustration of every closure for an arbitrary $2$-tangle $T$.  In the particular case of a $2$-tangle, notice that the two closures without virtual crossings correspond to the numerator closure and denominator closure of $T$.

\begin{figure}[ht!]
\centering
\scalebox{.75}{
\begin{tikzpicture}
[scale=1,auto=left,every node/.style={circle,inner sep=0pt}]
\draw[line width=3pt] (0,0) circle (1cm);
\node (T) at (0:0) {\Huge{\textbf{T}}};
\node (1) at (-90:1) {};
\node (2) at (-180:1) {};
\node (3) at (-270:1) {};
\node (4) at (-360:1) {};
\node (1*) at (-90:.75) {\small{1}};
\node (2*) at (-180:.75) {\small{2}};
\node (3*) at (-270:.75) {\small{3}};
\node (4*) at (-360:.75) {\small{4}};
\draw[line width=1.4pt] (1) circle arc (-20:-250:.82);
\draw[line width=1.4pt] (3) circle arc (-200:-430:.82);
\node (T1) at (90:2.75) {\Large{$\mathbf{T(m_1)}$}};
\end{tikzpicture}
\hspace{.5in}
\begin{tikzpicture}
[scale=1,auto=left,every node/.style={circle,inner sep=0pt}]
\draw[line width=3pt] (0,0) circle (1cm);
\node (T) at (0:0) {\Huge{\textbf{T}}};
\node (1) at (-90:1) {};
\node (2) at (-180:1) {};
\node (3) at (-270:1) {};
\node (4) at (-360:1) {};
\node (1*) at (-90:.75) {\small{1}};
\node (2*) at (-180:.75) {\small{2}};
\node (3*) at (-270:.75) {\small{3}};
\node (4*) at (-360:.75) {\small{4}};
\draw[line width=1.4pt] (1) circle arc (-160:70:.82);
\draw[line width=1.4pt] (3) circle arc (-340:-110:.82);
\node (T2) at (90:2.75) {\Large{$\mathbf{T(m_2)}$}};
\end{tikzpicture}
\hspace{.5in}
\raisebox{5pt}{
\begin{tikzpicture}
[scale=1,auto=left,every node/.style={circle,inner sep=0pt}]
\draw[line width=3pt] (0,0) circle (1cm);
\node (T) at (0:0) {\Huge{\textbf{T}}};
\node (1) at (-90:1) {};
\node (2) at (-180:1) {};
\node (3) at (-270:1) {};
\node (4) at (-360:1) {};
\node (1*) at (-90:.75) {\small{1}};
\node (2*) at (-180:.75) {\small{2}};
\node (3*) at (-270:.75) {\small{3}};
\node (4*) at (-360:.75) {\small{4}};
\draw[line width=1.4pt] (1) arc (-50:-310:1.31);
\draw[line width=1.4pt] (2) arc (-140:-400:1.31);
\node[draw,line width=.8pt,inner sep=4pt] (v) at (-225:1.74) {};
\node (T3) at (90:2.75) {\Large{$\mathbf{T(m_3)}$}};
\end{tikzpicture}}}

\vspace{.05in}

\caption{The three closures of the $2$-tangle $T$, corresponding to the matchings $m_1 = ((1,2),(3,4))$, $m_2 = ((1,4),(2,3))$, and $m_3 = ((1,3),(2,4))$ .}
\label{fig: 4-strand closures}
\end{figure}
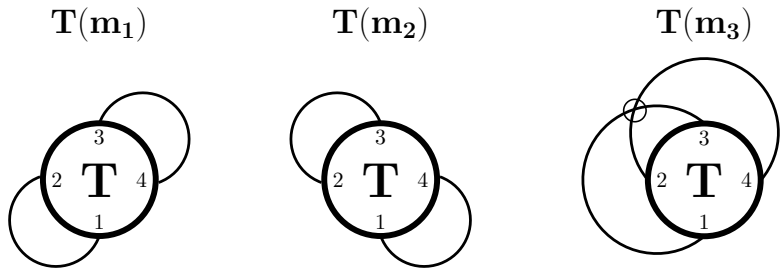

Similarly to how $\langle T \rangle$ may be written in terms of the $\langle T_m \rangle$, the bracket polynomial for the link $\langle T \cup T' \rangle$ may be written in terms of the $\langle T(m) \rangle$.  See Fish and Keyman \cite{FishKeyman} for a distinct derivation of a result equivalent to Proposition \ref{thm: bracket polynomial decomposition}.

\begin{proposition}
\label{thm: bracket polynomial decomposition}
Let $L = T \cup T'$ be any virtual link that has been decomposed into a pair of $n$-tangles $T$ and $T'$.  Then

$$\langle T \cup T' \rangle = \sum_{m \in \mathcal{P}_n} q_m \langle T(m) \rangle$$

\noindent where the $q_m \in \Z[A,A^{-1}]$ are Laurent polynomials that depend upon the structure of $T'$.
\end{proposition}

We wish to translate Proposition \ref{thm: bracket polynomial decomposition} into a result involving the auxiliary polynomial $f(T \cup T')$.  The difficulty is that this must be done in a way that doesn't require internal knowledge of $T$.  In particular, once we declare a specific orientation for $L = T \cup T'$, many closures $T(m)$ may fail to be compatible with that orientation.  Simply defining $f(T(m))$ in those cases would require a reorientation of some proper subset of the strands from $T$, an action whose effect on the writhe may require internal knowledge of $T$.

One way of avoiding this problem is to work with oriented and disoriented resolutions at a real crossing, so that the Kauffman states are purely virtual magnetic graphs.  See Kamada and Miyazawa \cite{Kamada} and Miyazawa \cite{Miyazawa} for results involving the resulting generalization of the Kauffman-Jones polynomial.  Unfortunately, working with purely virtual magnetic graphs significantly complicates much of what follows, and we instead use the Kauffman skein relation to systematically replace all problematic closures with diagrams that respect the desired orientation.

So consider any word $\vec{v}$ of length $2n$ that features exactly $n$ instances of $+$ and $n$ instances of $-$, and let $v_i$ denote the $i^{th}$ letter of $\vec{v}$.  We say that the $2n$-tangle $T$ has orientation $\vec{v}$ if its $i^{th}$ endpoint has an outbound orientation precisely when $v_i = +$.  For any tangle $T$ of orientation $\vec{v}$, we construct a braid $B_{\vec{v}}$ on $2n$ strands as follows:

\begin{enumerate}
\item Identify the longest initial subword $s$ of $\vec{v}$ that is of the form $(+-)^k$ or $(+-)^k+$.
\item If $\vert s \vert = 2n$, terminate the procedure.  If $\vert s \vert < 2n$, identify the smallest index $j > \vert s \vert$ such that $v_j \neq v_{\vert s \vert + 1}$ and add $\sigma_{j-1} \sigma_{j-2} \hdots \sigma_{\vert s \vert +1}$ to the end of $B_{\vec{v}}$.
\item Define $\vec{v}\kern+2pt '$ to be the length $2n$ word whose letters satisfy $v'_{\vert s \vert +1} = v_j$, $v'_i = v_{i-1}$ for $\vert s \vert + 2 \leq i \leq j$, and $v'_i = v_i$ otherwise.  Then return to Step \#1 using $\vec{v} = \vec{v}\kern+2pt '$.
\end{enumerate}

As the new word $\vec{v} \kern+2pt '$ in Step \#3 always has a longer initial subword of the required form than did $\vec{v}$, the procedure above terminates after a finite number of steps.  See Figure \ref{fig: n=2 and n=3 orientation braids} for the braids $B_{\vec{v}}$ associated with each (fundamentally distinct) orientation $\vec{v}$ on $4$ or $6$ endpoints.

\begin{figure}[ht!]
\centering
\begin{tikzpicture}
[scale=1,auto=left,every node/.style={circle,inner sep=0pt}]
	\draw[line width=1.4pt] (0.25,0) to (2.25,0);
	\draw[line width=1pt] (.5,0) to (.5,-1.2);
	\draw[line width=1pt] (1,0) to (1,-1.2);
	\draw[line width=1pt] (1.5,0) to (1.5,-1.2);
	\draw[line width=1pt] (2,0) to (2,-1.2);
	\draw[line width=1.4pt] (0.25,-1.2) to (2.25,-1.2);
	\node (a1) at (.5,-1.45) {\small{+}};
	\node (a2) at (1,-1.45) {\small{-}};
	\node (a3) at (1.5,-1.45) {\small{+}};
	\node (a4) at (2,-1.45) {\small{-}};
	\node (b1) at (.5,.25) {\small{+}};
	\node (b2) at (1,.25) {\small{-}};
	\node (b3) at (1.5,.25) {\small{+}};
	\node (b4) at (2,.25) {\small{-}};
	\node (label) at (1.25,-2) {$B_{\vec{v}} = \text{id}$};
\end{tikzpicture}
\hspace{.2in}
\raisebox{0.5pt}{
\begin{tikzpicture}
[scale=1,auto=left,every node/.style={circle,inner sep=0pt}]
	\draw[line width=1.4pt] (0.25,0) to (2.25,0);
	\braid[number of strands=4,line width=1pt,width=14pt,height=20pt] a_2;
	\draw[line width=1.4pt] (0.25,-1.2) to (2.25,-1.2);
	\node (a1) at (.5,-1.45) {\small{+}};
	\node (a2) at (1,-1.45) {\small{+}};
	\node (a3) at (1.5,-1.45) {\small{-}};
	\node (a4) at (2,-1.45) {\small{-}};
	\node (b1) at (.5,.25) {\small{+}};
	\node (b2) at (1,.25) {\small{-}};
	\node (b3) at (1.5,.25) {\small{+}};
	\node (b4) at (2,.25) {\small{-}};
	\node (label) at (1.25,-1.95) {$B_{\vec{v}} = \sigma_2$};
\end{tikzpicture}}

\begin{tikzpicture}
[scale=1,auto=left,every node/.style={circle,inner sep=0pt}]
	\draw[line width=1.4pt] (0.25,0) to (3.25,0);
	\draw[line width=1pt] (.5,0) to (.5,-1.8);
	\draw[line width=1pt] (1,0) to (1,-1.8);
	\draw[line width=1pt] (1.5,0) to (1.5,-1.8);
	\draw[line width=1pt] (2,0) to (2,-1.8);
	\draw[line width=1pt] (2.5,0) to (2.5,-1.8);
	\draw[line width=1pt] (3,0) to (3,-1.8);	
	\draw[line width=1.4pt] (0.25,-1.8) to (3.25,-1.8);
	\node (a1) at (.5,-2.05) {\small{+}};
	\node (a2) at (1,-2.05) {\small{-}};
	\node (a3) at (1.5,-2.05) {\small{+}};
	\node (a4) at (2,-2.05) {\small{-}};
	\node (a5) at (2.5,-2.05) {\small{+}};
	\node (a6) at (3,-2.05) {\small{-}};
	\node (b1) at (.5,.25) {\small{+}};
	\node (b2) at (1,.25) {\small{-}};
	\node (b3) at (1.5,.25) {\small{+}};
	\node (b4) at (2,.25) {\small{-}};
	\node (b5) at (2.5,.25) {\small{+}};
	\node (b6) at (3,.25) {\small{-}};
	\node (label) at (1.75,-2.55) {$B_{\vec{v}} = \text{id}$};
\end{tikzpicture}
\hspace{.2in}
\raisebox{-1pt}{
\begin{tikzpicture}
[scale=1,auto=left,every node/.style={circle,inner sep=0pt}]
	\draw[line width=1.4pt] (0.25,0) to (3.25,0);
	\braid[number of strands=6,line width=1pt,width=14pt,height=36pt] a_2;
	\draw[line width=1.4pt] (0.25,-1.8) to (3.25,-1.8);
	\node (a1) at (.5,-2.05) {\small{+}};
	\node (a2) at (1,-2.05) {\small{+}};
	\node (a3) at (1.5,-2.05) {\small{-}};
	\node (a4) at (2,-2.05) {\small{-}};
	\node (a5) at (2.5,-2.05) {\small{+}};
	\node (a6) at (3,-2.05) {\small{-}};
	\node (b1) at (.5,.25) {\small{+}};
	\node (b2) at (1,.25) {\small{-}};
	\node (b3) at (1.5,.25) {\small{+}};
	\node (b4) at (2,.25) {\small{-}};
	\node (b5) at (2.5,.25) {\small{+}};
	\node (b6) at (3,.25) {\small{-}};
	\node (label) at (1.75,-2.55) {$B_{\vec{v}} = \sigma_2$};
\end{tikzpicture}}
\hspace{.2in}
\raisebox{-7.5pt}{
\begin{tikzpicture}
[scale=1,auto=left,every node/.style={circle,inner sep=0pt}]
	\draw[line width=1.4pt] (0.25,0) to (3.25,0);
	\braid[number of strands=6,line width=1pt,width=14pt,height=18pt] a_4 a_2;
	\draw[line width=1.4pt] (0.25,-1.8) to (3.25,-1.8);
	\node (a1) at (.5,-2.05) {\small{+}};
	\node (a2) at (1,-2.05) {\small{+}};
	\node (a3) at (1.5,-2.05) {\small{-}};
	\node (a4) at (2,-2.05) {\small{+}};
	\node (a5) at (2.5,-2.05) {\small{-}};
	\node (a6) at (3,-2.05) {\small{-}};
	\node (b1) at (.5,.25) {\small{+}};
	\node (b2) at (1,.25) {\small{-}};
	\node (b3) at (1.5,.25) {\small{+}};
	\node (b4) at (2,.25) {\small{-}};
	\node (b5) at (2.5,.25) {\small{+}};
	\node (b6) at (3,.25) {\small{-}};
	\node (label) at (1.75,-2.55) {$B_{\vec{v}} = \sigma_2 \kern+2pt \sigma_4$};
\end{tikzpicture}}
\hspace{.2in}
\raisebox{-14pt}{
\begin{tikzpicture}
[scale=1,auto=left,every node/.style={circle,inner sep=0pt}]
	\draw[line width=1.4pt] (0.25,0) to (3.25,0);
	\braid[number of strands=6,line width=1pt,width=14pt,height=12pt] a_4 a_2 a_3;
	\draw[line width=1.4pt] (0.25,-1.8) to (3.25,-1.8);
	\node (a1) at (.5,-2.05) {\small{+}};
	\node (a2) at (1,-2.05) {\small{+}};
	\node (a3) at (1.5,-2.05) {\small{+}};
	\node (a4) at (2,-2.05) {\small{-}};
	\node (a5) at (2.5,-2.05) {\small{-}};
	\node (a6) at (3,-2.05) {\small{-}};
	\node (b1) at (.5,.25) {\small{+}};
	\node (b2) at (1,.25) {\small{-}};
	\node (b3) at (1.5,.25) {\small{+}};
	\node (b4) at (2,.25) {\small{-}};
	\node (b5) at (2.5,.25) {\small{+}};
	\node (b6) at (3,.25) {\small{-}};
	\node (label) at (1.75,-2.55) {$B_{\vec{v}} = \sigma_3 \kern+2pt \sigma_2 \kern+2pt \sigma_4$};
\end{tikzpicture}}

\vspace{-.4in}

\caption{Up to cyclic permutation, the distinct orientation braids $B_{\vec{v}}$ on $4$ and $6$ endpoints.}
\label{fig: n=2 and n=3 orientation braids}
\end{figure}
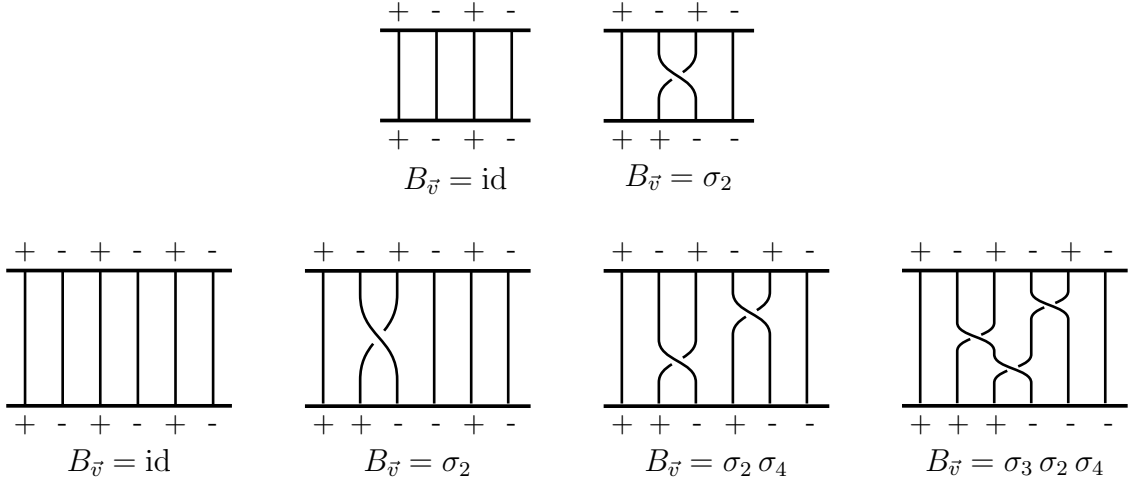

Now take the virtual link $L = T \cup T'$, and assume that $L$ has been oriented in such a way that $T$ has orientation $\vec{v}$.  Then identify the $i^{th}$ endpoint of $T$ with the bottom of the $i^{th}$ strand of $B_{\vec{v}}$.  This produces an oriented $n$-tangle $T^B$ whose endpoints alternate between inbound and outbound strands, a situation that we henceforth refer to as the ``standard orientation" for an $n$-tangle.  For any closure $T(m)$ of $T$, there exists an associated closure $T^B(m)$ of $T^B$ that is produced by attaching $T(m)-T$ to the endpoints of $T^B$.

As one final modification to ensure that our closures respect the orientation $\vec{v}$, we transform each closure $T^B(m)$ into $\widetilde{T}^B(m)$ by replacing the neighborhood of every virtual crossing in $T^B(m)-T^B$ as shown in Figure \ref{fig: virtual crossing replacement}.  The labels on the left side of that figure indicate the endpoints of $T^B$ to which each strand is eventually attached, where we assume that $a_1 < a_2 < a_3 < a_4$ to ensure that the local operation is well-defined.

\begin{figure}[ht!]
\centering
\scalebox{.85}{
\raisebox{2pt}{
\begin{tikzpicture}
[scale=1,auto=left,every node/.style={circle,inner sep=0pt}]
	\draw[line width=1.4pt] (-.5,-.65) to (.5,.65) {};
	\draw[line width=1.4pt] (.5,-.65) to (-.5,.65) {};
	\node[draw,line width=.8pt,inner sep=4pt] (v) at (0,0) {};
	\node (1*) at (.8,.8) {$a_1$};
	\node (2*) at (.8,-.8) {$a_2$};
	\node (3*) at (-.8,-.8) {$a_3$};
	\node (4*) at (-.8,.8) {$a_4$};
\end{tikzpicture}}
\hspace{.1in}
\raisebox{22pt}{
\scalebox{3}{$\Rightarrow$}}
\hspace{.15in}
\begin{tikzpicture}
[scale=1,auto=left,every node/.style={circle,inner sep=0pt}]
	\draw[line width=1.4pt,bend left=50] (-.1,-1.1) to (0,0) {};
	\draw[line width=1.4pt,bend right=50] (0,-1.2) to (0,0) {};
	\draw[line width=1.4pt] (0,0) to (-.5,.5) {};
	\draw[line width=1.4pt] (0,0) to (.5,.5) {};
	\draw[line width=1.4pt] (-.5,-1.7) to (0,-1.2) {};
	\draw[line width=1.4pt] (.5,-1.7) to (.1,-1.3) {};
	\node[draw,line width=.8pt,inner sep=4pt] (v) at (0,0) {};
\end{tikzpicture}}
\caption{Replacing the neighborhood of a virtual crossing in $T^B(m)-T^B$.}
\label{fig: virtual crossing replacement}
\end{figure}
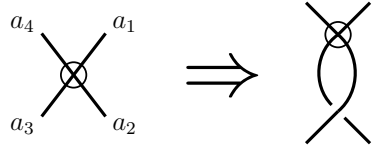

Lemma \ref{thm: auxiliary polynomial decomposition, lemma 1} shows that the original closures $T(m)$ from the decomposition of Proposition \ref{thm: bracket polynomial decomposition} may be swapped out for the modified closures $\widetilde{T}^B(m)$:

\begin{lemma}
\label{thm: auxiliary polynomial decomposition, lemma 1}
Let $T$ be an $n$-tangle with orientation $\vec{v}$, and take any $m \in \mathcal{P}_n$.  Then there exist $p_\mu \in \Z[A,A^{-1}]$ such that
$$\langle T(m) \rangle = \sum_{\mu \in \mathcal{P}_n} p_{\mu} \langle \widetilde{T}^B(\mu) \rangle$$
\end{lemma}
\begin{proof}
For any $m \in \mathcal{P}_n$, we first find $q_\mu \in \Z[A,A^{-1}]$ such that $\langle T(m) \rangle = \sum_{\mu \in \mathcal{P}_n} q_\mu \langle T^B(\mu) \rangle$.  For any $m \in \mathcal{P}_n$, we then provide $\widetilde{q}_\mu \in \Z[A,A^{-1}]$ such that $\langle T^B(m) \rangle = \sum_{\mu \in \mathcal{P}_n} \widetilde{q}_\mu \langle \widetilde{T}^B(\mu) \rangle$.

So assume $B_{\vec{v}} = \sigma_{i_1} \sigma_{i_2} \hdots \sigma_{i_M}$, and let $b_k = \sigma_{i_1} \sigma_{i_2} \hdots \sigma_{i_k}$ be the initial subword of $B_{\vec{v}}$ of length $k$.  For each $0 \leq k \leq M$, we define $T^{b_k}$ to be the $n$-tangle created by identifying the $i^{th}$ endpoint of $T$ with the bottom of the $i^{th}$ strand of $b_k$, so that $T^{b_0} = T$ and $T^{b_M} = T^B$.  For fixed $1 \leq k \leq M$ and any $m \in \mathcal{P}_n$ we demonstrate there exist $q_\mu \in \Z[A,A^{-1}]$ such that $\langle T^{b_{k-1}}(m) \rangle = \sum_{\mu \in \mathcal{P}_n} q_\mu \langle T^{b_k}(\mu) \rangle$.

Take $T^{b_k}(m)$, and consider the neighborhood of the final crossing $\sigma_k$ from $b_k$ in $T^{b_k}(m)$, located just inside the boundary of $T^{b_k}$.  The Kauffman-Jones skein relation gives the following, where the horizontal line denotes the external boundary of $T^{b_k}$ and the $-A^{\pm3}$ term is determined by the writhe of the nugatory crossing introduced on the right side.

\begin{center}
\raisebox{9pt}{\scalebox{2}{$\langle$}} \kern-4pt
\scalebox{.45}{\begin{tikzpicture}
[scale=1,auto=left,every node/.style={circle,inner sep=0pt}]
\draw[line width=1.4pt,bend right=20] (.1,-1.1) to (.5,.5) {};
\draw[line width=1.4pt,bend left=20] (0,-1.2) to (-.5,.5) {};
\draw[line width=1.4pt] (.5,-1.7) to (0,-1.2) {};
\draw[line width=1.4pt] (-.5,-1.7) to (-.1,-1.3) {};
\draw[dotted,thick] (-1,-.5) to (1,-.5) {};
\end{tikzpicture}} \kern-4pt
\raisebox{9pt}{\scalebox{2}{$\rangle$} $\ = \ A$ \scalebox{2}{$\langle$}} \kern-4pt
\scalebox{.45}{\begin{tikzpicture}
[scale=1,auto=left,every node/.style={circle,inner sep=0pt}]
\draw[line width=1.4pt,bend left=80] (-.5,-1.7) to (.5,-1.7) {};
\draw[line width=1.4pt,bend right=80] (-.5,-.7) to (.5,-.7) {};
\draw[line width=1.4pt] (-.5,-.7) to (-.5,.5) {};
\draw[line width=1.4pt] (.5,-.7) to (.5,.5) {};
\draw[dotted,thick] (-1,-.5) to (1,-.5) {};
\end{tikzpicture}} \kern-4pt
\raisebox{9pt}{\scalebox{2}{$\rangle$} $\ + \ A^{-1}$ \scalebox{2}{$\langle$}} \kern-4pt
\scalebox{.45}{\begin{tikzpicture}
[scale=1,auto=left,every node/.style={circle,inner sep=0pt}]
\draw[line width=1.4pt,bend right=15] (-.5,-1.7) to (-.5,.5) {};
\draw[line width=1.4pt,bend left=15] (.5,-1.7) to (.5,.5) {};
\draw[dotted,thick] (-1,-.5) to (1,-.5) {};
\end{tikzpicture}} \kern-4pt
\raisebox{9pt}{\scalebox{2}{$\rangle$} $\ = \ A (-A^{\pm3})$ \scalebox{2}{$\langle$}} \kern-4pt
\scalebox{.45}{\begin{tikzpicture}
[scale=1,auto=left,every node/.style={circle,inner sep=0pt}]
\draw[line width=1.4pt,bend right=10] (.1,-1.1) to (.5,-.4) {};
\draw[line width=1.4pt,bend left=10] (0,-1.2) to (-.5,-.4) {};
\draw[line width=1.4pt] (.5,-1.7) to (0,-1.2) {};
\draw[line width=1.4pt] (-.5,-1.7) to (-.1,-1.3) {};
\draw[line width=1.4pt,bend right=80] (.5,-.4) to (-.5,-.4) {};
\draw[line width=1.4pt,bend left=80] (.5,.5) to (-.5,.5) {};
\draw[dotted,thick] (-1,-.5) to (1,-.5) {};
\end{tikzpicture}} \kern-4pt
\raisebox{9pt}{\scalebox{2}{$\rangle$} $\ + \ A^{-1}$ \scalebox{2}{$\langle$}} \kern-4pt
\scalebox{.45}{\begin{tikzpicture}
[scale=1,auto=left,every node/.style={circle,inner sep=0pt}]
\draw[line width=1.4pt,bend right=15] (-.5,-1.7) to (-.5,.5) {};
\draw[line width=1.4pt,bend left=15] (.5,-1.7) to (.5,.5) {};
\draw[dotted,thick] (-1,-.5) to (1,-.5) {};
\end{tikzpicture}} \kern-4pt
\raisebox{10pt}{\scalebox{2}{$\rangle$}}

\vspace{.15in}

\raisebox{8pt}{\scalebox{2.5}{$\Rightarrow$}} \hspace{.1in}
\raisebox{9pt}{\scalebox{2}{$\langle$}} \kern-4pt
\scalebox{.45}{\begin{tikzpicture}
[scale=1,auto=left,every node/.style={circle,inner sep=0pt}]
\draw[line width=1.4pt,bend right=15] (-.5,-1.7) to (-.5,.5) {};
\draw[line width=1.4pt,bend left=15] (.5,-1.7) to (.5,.5) {};
\draw[dotted,thick] (-1,-.5) to (1,-.5) {};
\end{tikzpicture}} \kern-4pt
\raisebox{9pt}{\scalebox{2}{$\rangle$} $\ = \ A$ \scalebox{2}{$\langle$}} \kern-4pt
\scalebox{.45}{\begin{tikzpicture}
[scale=1,auto=left,every node/.style={circle,inner sep=0pt}]
\draw[line width=1.4pt,bend right=20] (.1,-1.1) to (.5,.5) {};
\draw[line width=1.4pt,bend left=20] (0,-1.2) to (-.5,.5) {};
\draw[line width=1.4pt] (.5,-1.7) to (0,-1.2) {};
\draw[line width=1.4pt] (-.5,-1.7) to (-.1,-1.3) {};
\draw[dotted,thick] (-1,-.5) to (1,-.5) {};
\end{tikzpicture}} \kern-4pt
\raisebox{9pt}{\scalebox{2}{$\rangle$} $\ - \ A^{2} (-A^{\pm3})$ \scalebox{2}{$\langle$}} \kern-4pt
\scalebox{.45}{\begin{tikzpicture}
[scale=1,auto=left,every node/.style={circle,inner sep=0pt}]
\draw[line width=1.4pt,bend right=10] (.1,-1.1) to (.5,-.4) {};
\draw[line width=1.4pt,bend left=10] (0,-1.2) to (-.5,-.4) {};
\draw[line width=1.4pt] (.5,-1.7) to (0,-1.2) {};
\draw[line width=1.4pt] (-.5,-1.7) to (-.1,-1.3) {};
\draw[line width=1.4pt,bend right=80] (.5,-.4) to (-.5,-.4) {};
\draw[line width=1.4pt,bend left=80] (.5,.5) to (-.5,.5) {};
\draw[dotted,thick] (-1,-.5) to (1,-.5) {};
\end{tikzpicture}} \kern-4pt
\raisebox{9pt}{\scalebox{2}{$\rangle$}}
\end{center}

Notice that the first term in the second equation is simply $\langle T^{b_{k-1}}(m) \rangle$.  Since the diagram associated with the final term of the second equation lacks classical crossings away from $T^{b_k}$, after the removal of trivial split components and nugatory crossings it must be equivalent to $T^{b_k}(\mu)$ for some $\mu \in \mathcal{P}_n$.  Thus $\langle T^{b_{k-1}}(m) \rangle = A \kern+1pt \langle T^{b_k}(m) \rangle - A^2 (-A^{-2}-A^{2})^{t_1} (-A^3)^{t_2} \kern+1pt \langle T^{b_k}(\mu) \rangle$ for some $t_1 \geq 0$, $t_2 \in \Z$, and $\mu \in \mathcal{P}_n$.  Repeatedly applying this result until reaching $k = M$ allows us to conclude that $\langle T(m) \rangle = \sum_{\mu \in \mathcal{P}_n} q_\mu \langle T^B(\mu) \rangle$ for some $q_\mu \in \mathcal{P}_n$.

Now consider the set of closures $\lbrace T^B(m) \rbrace_{m \in \mathcal{P}_n}$, and let $S_k$ denote the subset of those links that feature precisely $k$ virtual crossings away from $T$.  We induct on $k \geq 0$, showing that any $T^B(m) \in S_k$ may be written as $\langle T^B(m) \rangle = \sum_{\mu \in \mathcal{P}_n} p_{m,\mu} \langle \widetilde{T}^B(\mu) \rangle$ for some $p_{m,\mu} \in \Z[A,A^{-1}]$.

The case of $k = 0$ follows from the fact that $T^B(m) = \widetilde{T}^B(m)$ for any closure that lacks virtual crossings away from $T^B$.  So take any $T^B(m) \in S_k$, where $k \geq 1$, and consider the associated link $\widetilde{T}^B(m)$.  In the neighborhood of any classical crossing in $\widetilde{T}^B(m) - T^B$, the Kauffman-Jones skein relation gives

\begin{center}
\raisebox{9pt}{\scalebox{2}{$\langle$}}
\scalebox{.45}{\begin{tikzpicture}
[scale=1,auto=left,every node/.style={circle,inner sep=0pt}]
\draw[line width=1.4pt,bend left=50] (-.1,-1.1) to (0,0) {};
\draw[line width=1.4pt,bend right=50] (0,-1.2) to (0,0) {};
\draw[line width=1.4pt] (0,0) to (-.5,.5) {};
\draw[line width=1.4pt] (0,0) to (.5,.5) {};
\draw[line width=1.4pt] (-.5,-1.7) to (0,-1.2) {};
\draw[line width=1.4pt] (.5,-1.7) to (.1,-1.3) {};
\node[draw,line width=.8pt,inner sep=4pt] (v) at (0,0) {};
\end{tikzpicture}}
\raisebox{9pt}{\scalebox{2}{$\rangle$} $\ = \ A$ \scalebox{2}{$\langle$}}
\scalebox{.45}{\begin{tikzpicture}
[scale=1,auto=left,every node/.style={circle,inner sep=0pt}]
\draw[line width=1.4pt,bend left=20] (-.5,-1.7) to (0,0) {};
\draw[line width=1.4pt,bend right=20] (.5,-1.7) to (0,0) {};
\draw[line width=1.4pt] (0,0) to (-.5,.5) {};
\draw[line width=1.4pt] (0,0) to (.5,.5) {};
\node[draw,line width=.8pt,inner sep=4pt] (v) at (0,0) {};
\end{tikzpicture}}
\raisebox{9pt}{\scalebox{2}{$\rangle$} $\ + \ A^{-1}$ \scalebox{2}{$\langle$}}
\scalebox{.45}{\begin{tikzpicture}
[scale=1,auto=left,every node/.style={circle,inner sep=0pt}]
\draw[line width=1.4pt,bend left=80] (-.5,-1.7) to (.5,-1.7) {};
\draw[line width=1.4pt,bend right=60] (0,0) to (-.1,-1) {};
\draw[line width=1.4pt,bend left=60] (0,0) to (.1,-1) {};
\draw[line width=1.4pt,bend right=40] (-.1,-1) to (.1,-1) {};
\draw[line width=1.4pt] (0,0) to (-.5,.5) {};
\draw[line width=1.4pt] (0,0) to (.5,.5) {};
\node[draw,line width=.8pt,inner sep=4pt] (v) at (0,0) {};
\end{tikzpicture}}
\raisebox{9pt}{\scalebox{2}{$\rangle$} $\ = \ A$ \scalebox{2}{$\langle$}}
\raisebox{6pt}{\scalebox{.45}{\begin{tikzpicture}
[scale=1,auto=left,every node/.style={circle,inner sep=0pt}]
\draw[line width=1.4pt] (-.5,-.65) to (.5,.65) {};
\draw[line width=1.4pt] (.5,-.65) to (-.5,.65) {};
\node[draw,line width=.8pt,inner sep=4pt] (v) at (0,0) {};
\end{tikzpicture}}}
\raisebox{9pt}{\scalebox{2}{$\rangle$} $\ + \ A^{-1}$ \scalebox{2}{$\langle$}}
\raisebox{6pt}{\scalebox{.45}{\begin{tikzpicture}
[scale=1,auto=left,every node/.style={circle,inner sep=0pt}]
\draw[bend left=80,line width=1.4pt] (-.5,-.65) to (.5,-.65) {};
\draw[bend right=80,line width=1.4pt] (-.5,.65) to (.5,.65) {};
\end{tikzpicture}}}
\raisebox{10pt}{\scalebox{2}{$\rangle$}}
\end{center}

Resolving every classical crossing of $\widetilde{T}^B(m) - T^B$ as above gives $\langle \widetilde{T}^B(m) \rangle = A^k \langle T^B(m) \rangle + \sum_\alpha q_\alpha \langle D_\alpha \rangle$ for $q_\alpha \in \Z[A,A^{-1}]$ and some collection of link diagrams $D_\alpha$, each of which contain $T^B$, lack classical crossings away from $T^B$, and have at most $k-1$ virtual crossings away from $T^B$.   Up to trivial split components and nugatory crossings, each $D_\alpha$ is then equivalent to some closure $T^B(m_\alpha)$ that has at most $k-1$ virtual crossings away from $T^B$.  It follows that $\langle \widetilde{T}^B(m) \rangle = A^k \langle T^B(m) \rangle + \sum_{m \in \mathcal{P}_n} p'_m \kern+1pt \langle T^B(m_\alpha) \rangle$ for some $p'_m \in \Z[A,A^{-1}]$ and some set of closures $T^B(m_\alpha)$ that each contain at most $k-1$ virtual crossings away from $T^B$.  Rearranging gives $\langle T^B(m) \rangle = A^{-k} \langle \widetilde{T}^B(m) \rangle - A^{-k}\sum_{m \in \mathcal{P}_n} p'_m \kern+1pt \langle T^B(m_\alpha) \rangle$ for some set of closures $T^B(m_\alpha)$ that each contain at most $k-1$ virtual crossings away from $T^B$.  Applying the inductive assumption allows us to conclude $\langle T^B(m) \rangle = \sum_{\mu \in \mathcal{P}_n} \widetilde{q}_\mu \langle \widetilde{T}^B(\mu) \rangle$ for some $\widetilde{q_{\mu}} \in \Z[A,A^{-1}]$.
\end{proof}

Pause to note that Lemma \ref{thm: auxiliary polynomial decomposition, lemma 1} is dependent upon the specific algorithm by which we transformed each $T(m)$ into $\widetilde{T}^B(m)$.  That algorithm is certainly only one of many ways to systematically replace all closures with counterparts that are compatible with the given orientation.  It is an open question as to whether the divisors derived in Section \ref{sec: local moves} are identical to those that would result from a different definition of $\widetilde{T}^B(m)$.

We are now ready for the primary theorem of this section, which translates the decomposition of $\langle T \cup T' \rangle$ from Proposition \ref{thm: bracket polynomial decomposition} to a decomposition of the auxiliary polynomial $f(T \cup T')$, no matter the orientation on $T \cup T'$.

\begin{theorem}
\label{thm: auxiliary polynomial decomposition}
Let $L = T \cup T'$ be an oriented virtual link that has been decomposed into the $n$-tangles $T$ and $T'$.  Then

$$f(T \cup T') = \sum_{m \in \mathcal{P}_n} q_m f(\widetilde{T}^B(m))$$

\noindent where the $q_m \in \Z[A,A^{-1}]$ are Laurent polynomials that depend upon the structure of $T'$.
\end{theorem}
\begin{proof}
Applying Lemma \ref{thm: auxiliary polynomial decomposition, lemma 1} to Proposition \ref{thm: bracket polynomial decomposition}, we immediately know that $\langle T \cup T' \rangle = \sum_{m \in \mathcal{P}_n} p_m \langle \widetilde{T}^B(m) \rangle$ for some $p_m \in \Z[A,A^{-1}]$.  In order to translate this result to $f(T \cup T')$, we need to show that $\widetilde{T}^B(m)$ respects the orientation on $L$ for every $m \in \mathcal{P}_n$.

Begin by observing that, no matter the original orientation on $L$, the modified tangle $T^B$ always has endpoints that alternate between inward and outward strands.  As $T^B$ has standard orientation, it is straightforward to show that a particular closure $T^B(m)$ respects the orientation on $T^B$ if and only if $m$ is non-crossing.  Thus $\widetilde{T}^B(m) = T^B(m)$ respects the given orientation for all $m \in \mathcal{M}_n$.  For matchings $m \in \mathcal{P}_n$ with at least one crossing, consider the virtual link $L_m$ that may be obtained from $\widetilde{T}^B(m)$ by replacing the neighborhood of every virtual crossing in $\widetilde{T}^B(m)$ as shown below.

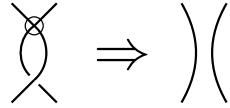
\begin{center}
\scalebox{.6}{
\begin{tikzpicture}
[scale=1,auto=left,every node/.style={circle,inner sep=0pt}]
\draw[line width=1.4pt,bend left=50] (-.1,-1.1) to (0,0) {};
\draw[line width=1.4pt,bend right=50] (0,-1.2) to (0,0) {};
\draw[line width=1.4pt] (0,0) to (-.5,.5) {};
\draw[line width=1.4pt] (0,0) to (.5,.5) {};
\draw[line width=1.4pt] (-.5,-1.7) to (0,-1.2) {};
\draw[line width=1.4pt] (.5,-1.7) to (.1,-1.3) {};
\node[draw,line width=.8pt,inner sep=4pt] (v) at (0,0) {};
\end{tikzpicture}
\hspace{.15in}
\raisebox{22pt}{
\scalebox{3}{$\Rightarrow$}}
\hspace{.15in}
\begin{tikzpicture}
[scale=1,auto=left,every node/.style={circle,inner sep=0pt}]
\draw[line width=1.4pt,bend right=30] (-.5,-1.7) to (-.5,.5) {};
\draw[line width=1.4pt,bend left=30] (.5,-1.7) to (.5,.5) {};
\end{tikzpicture}}
\end{center}

As each $L_m$ lacks crossings away from $T^B$, up to trivial split components it is equivalent to $T^B(m)$ for some $m \in \mathcal{M}_n$.  It follows that $L_m$ respects the orientation on $T^B$.  As the local move shown above is always orientation-preserving, we may conclude that $\widetilde{T}^B(m)$ respects the given orientation for any $m \in \mathcal{P}_n - \mathcal{M}_n$.

Now assume that our original $n$-tangles have writhes $w(T)=w$ and $w(T')=w'$.  We then have $\langle T \cup T' \rangle = (-A^3)^{w+w'} f(T \cup T')$.  Knowing that every modified closure $\widetilde{T}^B(m)$ is compatible with the orientation on $T \cup T'$, we also have $\langle \widetilde{T}^B(m) \rangle = (-A^3)^{w+\widetilde{w}_m} f(\widetilde{T}^B(m))$ for every $m \in \mathcal{P}_m$, where $\widetilde{w}_m$ is dependent upon the structure of $\widetilde{T}^B(m)-T$.  The theorem follows by substituting these results into $\langle T \cup T' \rangle = \sum_{m \in \mathcal{P}_n} p_m \langle \widetilde{T}^B(m) \rangle$.
\end{proof}

See Figure \ref{fig: n=2 tangle decompositions} for an illustration of the modified closures $\widetilde{T}^B(m)$ from Theorem \ref{thm: auxiliary polynomial decomposition}, when $T$ is a $2$-tangle with either of the orientations from Figure \ref{fig: n=2 and n=3 orientation braids}.

If the original link $L = T \cup T'$ lacks virtual crossings, observe that the closures $T^B(m)$ involving $m \in \mathcal{P}_n - \mathcal{M}_n$ never contribute to the summation of Theorem \ref{thm: auxiliary polynomial decomposition}.  Since we also have $\widetilde{T}^B(m) = T^B(m)$ for every $m \in \mathcal{M}_n$, we draw the following corollary.

\begin{corollary}
\label{thm: auxiliary polynomial decomposition, corollary}
Let $L = T \cup T'$ be an oriented classical link that has been decomposed into the $n$-tangles $T$ and $T'$.  Then

$$f(T \cup T') = \sum_{m \in \mathcal{M}_n} q_m f(T^B(m))$$

\noindent where the $q_m \in \Z[A,A^{-1}]$ are Laurent polynomials that depend upon the structure of $T'$.
\end{corollary}

\begin{figure}[ht!]
\centering
\scalebox{.75}{
\begin{tikzpicture}
[scale=1,auto=left,every node/.style={circle,inner sep=0pt}]
\draw[line width=3pt] (0,0) circle (1cm);
\node (T) at (0:0) {\Huge{\textbf{T}}};
\node (1) at (-90:1) {};
\node (2) at (-180:1) {};
\node (3) at (-270:1) {};
\node (4) at (-360:1) {};
\node (1*) at (-90:.75) {\small{1}};
\node (2*) at (-180:.75) {\small{2}};
\node (3*) at (-270:.75) {\small{3}};
\node (4*) at (-360:.75) {\small{4}};
\draw[line width=1.4pt,->-] (1) circle arc (-20:-250:.82);
\draw[line width=1.4pt,->-] (3) circle arc (-200:-430:.82);
\end{tikzpicture}
\hspace{.75in}
\begin{tikzpicture}
[scale=1,auto=left,every node/.style={circle,inner sep=0pt}]
\draw[line width=3pt] (0,0) circle (1cm);
\node (T) at (0:0) {\Huge{\textbf{T}}};
\node (1) at (-90:1) {};
\node (2) at (-180:1) {};
\node (3) at (-270:1) {};
\node (4) at (-360:1) {};
\node (1*) at (-90:.75) {\small{1}};
\node (2*) at (-180:.75) {\small{2}};
\node (3*) at (-270:.75) {\small{3}};
\node (4*) at (-360:.75) {\small{4}};
\draw[line width=1.4pt,->-] (1) circle arc (-160:70:.82);
\draw[line width=1.4pt,->-] (3) circle arc (-340:-110:.82);
\end{tikzpicture}
\hspace{.75in}
\raisebox{5pt}{
\begin{tikzpicture}
[scale=1,auto=left,every node/.style={circle,inner sep=0pt}]
\draw[line width=3pt] (0,0) circle (1cm);
\node (T) at (0:0) {\Huge{\textbf{T}}};
\node (1) at (-90:1) {};
\node (2) at (-180:1) {};
\node (3) at (-270:1) {};
\node (4) at (-360:1) {};
\node (1*) at (-90:.75) {\small{1}};
\node (2*) at (-180:.75) {\small{2}};
\node (3*) at (-270:.75) {\small{3}};
\node (4*) at (-360:.75) {\small{4}};
\draw[line width=1.4pt,->-] (1) arc (-50:-248:1.51);
\draw[line width=1.4pt,-<-] (4) arc (-400:-202:1.51);
\draw[line width=1.4pt, bend left=75] (-225:2.18) to (-225:1.5);
\draw[line width=1.4pt, bend right=75] (-225:2.18) to (-222:1.55);
\draw[line width=1.4pt, bend right=25] (-225:1.5) to (-180:1);
\draw[line width=1.4pt, bend left=25] (-230:1.48) to (-270:1);
\node[draw,line width=.8pt,inner sep=4pt] (v) at (-225:2.18) {};
\end{tikzpicture}}}

\vspace{.3in}

\scalebox{.75}{
\begin{tikzpicture}
[scale=1,auto=left,every node/.style={circle,inner sep=0pt}]
\draw[line width=3pt] (0,0) circle (1cm);
\node (T) at (0:0) {\Huge{\textbf{T}}};
\node (1) at (-90:1) {};
\node (2) at (-180:1) {};
\node (3) at (-270:1) {};
\node (4) at (-360:1) {};
\node (1*) at (-90:.75) {\small{1}};
\node (2*) at (-180:.75) {\small{2}};
\node (3*) at (-270:.75) {\small{3}};
\node (4*) at (-360:.75) {\small{4}};
\draw[line width=1.4pt,-<-] (4) arc (-45:165:1.23);
\draw[line width=1.4pt,-<-] (1) arc (-45:-260:1.23);
\draw[line width=1.4pt, bend right=15] (-220:1.48) to (-180:1);
\draw[line width=1.4pt, bend left=15] (-225:1.536) to (-270:1);
\end{tikzpicture}
\hspace{.75in}
\begin{tikzpicture}
[scale=1,auto=left,every node/.style={circle,inner sep=0pt}]
\draw[line width=3pt] (0,0) circle (1cm);
\node (T) at (0:0) {\Huge{\textbf{T}}};
\node (1) at (-90:1) {};
\node (2) at (-180:1) {};
\node (3) at (-270:1) {};
\node (4) at (-360:1) {};
\node (1*) at (-90:.75) {\small{1}};
\node (2*) at (-180:.75) {\small{2}};
\node (3*) at (-270:.75) {\small{3}};
\node (4*) at (-360:.75) {\small{4}};
\draw[line width=1.4pt,-<-] (-225:1.536) arc (-60:-401:.5);
\draw[line width=1.4pt, bend right=35] (-223:1.48) to (-180:1);
\draw[line width=1.4pt, bend left=15] (-225:1.536) to (-270:1);
\draw[line width=1.4pt,->-] (1) circle arc (-160:70:.82);
\end{tikzpicture}
\hspace{.03in}
\raisebox{42pt}{\scalebox{2}{$=$}}
\hspace{.03in}
\begin{tikzpicture}
[scale=1,auto=left,every node/.style={circle,inner sep=0pt}]
\draw[line width=3pt] (0,0) circle (1cm);
\node (T) at (0:0) {\Huge{\textbf{T}}};
\node (1) at (-90:1) {};
\node (2) at (-180:1) {};
\node (3) at (-270:1) {};
\node (4) at (-360:1) {};
\node (1*) at (-90:.75) {\small{1}};
\node (2*) at (-180:.75) {\small{2}};
\node (3*) at (-270:.75) {\small{3}};
\node (4*) at (-360:.75) {\small{4}};
\draw[line width=1.4pt,->-] (1) circle arc (-160:70:.82);
\draw[line width=1.4pt,-<-] (3) circle arc (-340:-110:.82);
\end{tikzpicture}}

\scalebox{.75}{
\begin{tikzpicture}
[scale=1,auto=left,every node/.style={circle,inner sep=0pt}]
\draw[line width=3pt] (0,0) circle (1cm);
\node (T) at (0:0) {\Huge{\textbf{T}}};
\node (1) at (-90:1) {};
\node (2) at (-180:1) {};
\node (3) at (-270:1) {};
\node (4) at (-360:1) {};
\node (1*) at (-90:.75) {\small{1}};
\node (2*) at (-180:.75) {\small{2}};
\node (3*) at (-270:.75) {\small{3}};
\node (4*) at (-360:.75) {\small{4}};
\draw[line width=1.4pt,->-] (1) arc (-50:-243:1.85);
\draw[line width=1.4pt,-<-] (4) arc (-400:-207:1.85);
\draw[line width=1.4pt, bend left=50] (-225:2.88) to (-225:2.1);
\draw[line width=1.4pt, bend right=50] (-225:2.88) to (-223:2.15);
\draw[line width=1.4pt, bend right=50] (-225:2.1) to (-225:1.5);
\draw[line width=1.4pt, bend left=27] (-225:1.5) to (-270:1);
\draw[line width=1.4pt, bend right=30] (-222:1.48) to (-180:1);
\draw[line width=1.4pt, bend left=50] (-226:2.03) to (-226:1.57);
\node[draw,line width=.8pt,inner sep=4pt] (v) at (-225:2.88) {};
\end{tikzpicture}
\hspace{.03in}
\raisebox{42pt}{\scalebox{2}{$=$}}
\hspace{.03in}
\begin{tikzpicture}
[scale=1,auto=left,every node/.style={circle,inner sep=0pt}]
\draw[line width=3pt] (0,0) circle (1cm);
\node (T) at (0:0) {\Huge{\textbf{T}}};
\node (1) at (-90:1) {};
\node (2) at (-180:1) {};
\node (3) at (-270:1) {};
\node (4) at (-360:1) {};
\node (1*) at (-90:.75) {\small{1}};
\node (2*) at (-180:.75) {\small{2}};
\node (3*) at (-270:.75) {\small{3}};
\node (4*) at (-360:.75) {\small{4}};
\draw[line width=1.4pt,->-] (1) arc (-50:-310:1.31);
\draw[line width=1.4pt,->-] (2) arc (-140:-400:1.31);
\node[draw,line width=.8pt,inner sep=4pt] (v) at (-225:1.74) {};
\end{tikzpicture}}

\caption{The three modified closures $\widetilde{T}^B(m)$ for a 2-tangle with orientation $\vec{v}=+-+-$ (row one) and a 2-tangle with orientation $\vec{v}=++--$ (rows two and three).}
\label{fig: n=2 tangle decompositions}
\end{figure}
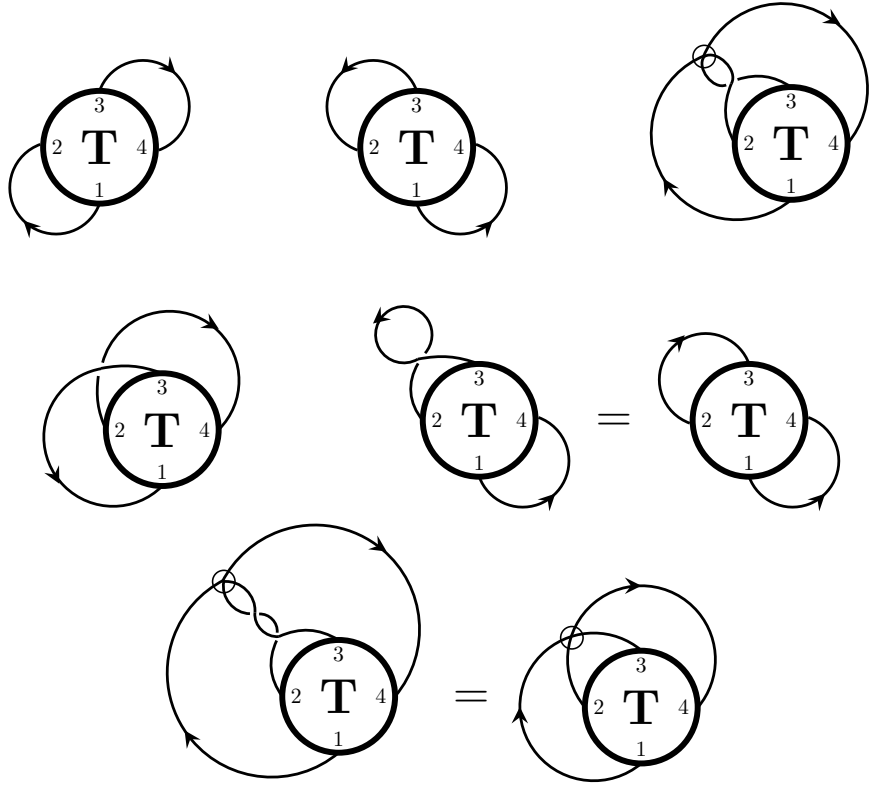

\section{Local Moves \& Divisibility of the Jones Polynomial}
\label{sec: local moves}

As our primary application of Theorem \ref{thm: auxiliary polynomial decomposition}, we investigate how various local moves effect the auxiliary polynomial of a virtual link.  So let $T_1$ be an $n$-tangle with orientation $\vec{v}$, and let $\phi$ be a local move that replaces $T_1$ with another $n$-tangle $\phi(T_1) = T_2$ of orientation $\vec{v}$.  Then consider the virtual links $L = T_1 \cup T'$ and $\phi(L) = T_2 \cup T'$, where $T'$ is an arbitrary $n$-tangle that has been oriented so as to be compatible with $T_1$.  Applying Theorem \ref{thm: auxiliary polynomial decomposition} to both of these links immediately yields the following:

\begin{proposition}
\label{thm: general auxiliary polynomial divisibility}
Let $\phi$ be a local move that replaces the oriented $n$-tangle $T_1$ with an $n$-tangle $T_2$ whose endpoints are equivalently oriented.  For any $n$-tangle $T'$, the polynomial $f(T_2 \cup T') - f(T_1 \cup T')$ is divisible by the gcd of the set $\lbrace f(\widetilde{T}_2^B(m))- f(\widetilde{T}_1^B(m)) \rbrace_{m \in \mathcal{P}_n}$.
\end{proposition}

Even for small $n$, Proposition \ref{thm: general auxiliary polynomial divisibility} is of limited usage unless one may succinctly characterize the closures $\widetilde{T}_1^B(m)$, $\widetilde{T}_2^B(m)$ for every $m \in \mathcal{P}_n$.  In what follows, we consider various classes of local moves where this characterization is tractable, restricting from virtual links to classical links as needed.  These classes will encompass, as special cases, oriented versions for many of the local moves considered by Ganzell \cite{Ganzell}.

Pause to observe that, if two links $L_1,L_2$ are related via a finite sequence of $\phi$-moves, repeated application of Proposition \ref{thm: general auxiliary polynomial divisibility} says that $f(L_1) - f(L_2)$ must be divisible by the greatest common divisor of the $\lbrace f(\widetilde{T}_2^B(m))- f(\widetilde{T}_1^B(m)) \rbrace_{m \in \mathcal{P}_n}$.  In cases where the $\widetilde{T}_1^B(m)$, $\widetilde{T}_2^B(m)$ are easily computable for all $m \in \mathcal{P}_n$, this provides a necessary condition for determining whether $\phi$ represents an unlinking operation:

\begin{corollary}
\label{thm: unlinking moves from auxiliary polynomial divisibility}
Let $\phi$ be a local move that replaces the oriented $n$-tangle $T_1$ with an $n$-tangle $T_2$ whose endpoints are equivalently oriented, and let $\bigcirc^k$ be the unlink of $k$ components. Then the virtual link $L$ may be transformed into $\bigcirc^k$ via a finite sequence of $\phi$-moves only if $f(L) - f(\bigcirc^k) = f(L) - (-A^{-2} - A^2)^{k-1}$ is divisible by the gcd of the $\lbrace f(\widetilde{T}_2^B(m))- f(\widetilde{T}_1^B(m)) \rbrace_{m \in \mathcal{P}_n}$.
\end{corollary}

\subsection{Rotational Local Moves of Classical $n$-tangles}
\label{subsec: rotational local moves}

In this subsection we restrict our attention to local moves $T_1 \mapsto T_2$ that involve rotation of $T_1$ by some fixed angle.  So let $T$ be an $n$-tangle, and let $r^k(T)$ be the $n$-tangle that results from rotating $T$ by $\frac{k\pi}{n}$ radians in the clockwise direction.  For every $T$ and every $k > 0$, this defines a local move $T \mapsto r^k(T)$ that may or may not preserve orientation on endpoints.

To ensure that our local moves preserve orientation, we restrict our attention to $n$-tangles with standard orientation $\vec{v}=(+-)^n$ and modify the rotational operator as follows.  If $k$ is even, $r^k(T)$ already has standard orientation and we define $\rho^k(T)=r^k(T)$.  If $k$ is odd, $r^k(T)$ has the opposite orientation $\vec{v}' = (-+)^n$.  In this case we define $\rho^k(T)$ to be the $n$-tangle of orientation $\vec{v}$ that results from reversing every strand in $r^k(T)$, including closed strands that do not terminate at the boundary.  Observe that $f(r^k(T)) = f(\rho^k(T))$ for any $T$, as reversing every strand in a tangle fixes the writhe of every crossing.

Now take $m \in \mathcal{P}_n$.  For every $k \geq 0$ we similarly define $r^k(m)$ to be the matching that results from a counterclockwise rotation of $m$ by $\frac{k\pi}{n}$ radians.  Clearly $r^k(m)$ is noncrossing for all $k>0$ if and only if $m$ is noncrossing.  For all $m \in \mathcal{P}_n$, notice that the closure $T(r^k(m))$ may be obtained from $T(m)$ via a counterclockwise rotation of $T(m)-T$ by $\frac{k\pi}{n}$ radians.

If $T$ has standard orientation, we immediately have $\rho^k(T)(m) = T(r^k(m))$ for all $m \in \mathcal{P}_n$ and all $k > 0$.  However, for matchings with crossings this equality breaks down when we replace the neighborhood of virtual crossings as in Figure \ref{fig: virtual crossing replacement}.  If we further assume that $m$ is noncrossing, we always have $\widetilde{T}^B(m) = T^B(m) =T(m)$ and may still assert $\widetilde{\rho^k(T)}\kern-0pt ^B(m) = \widetilde{T}^B(r^k(m))$ for every $k > 0$.  All of this gives the following specialization of Proposition \ref{thm: general auxiliary polynomial divisibility}.

\begin{proposition}
\label{thm: divisibility under rotational move}
Let $T$ be a classical $n$-tangle with standard orientation.  For any classical $n$-tangle $T'$ and every $k > 0$, the polynomial $f(\rho^k(T) \cup T') - f(T \cup T')$  is divisible by the gcd of the set $\lbrace f(T(r^k(m))) - f(T(m)) \rbrace_{m \in \mathcal{M}_n}$.
\end{proposition}

Looking to apply Proposition \ref{thm: divisibility under rotational move}, we first consider rotational local moves $T \mapsto \rho^k(T)$ where $T$ is a classical $2$-tangle.  Here we denote the two non-crossing matchings on $4$ points by $m_1 = ((1,2),(3,4))$ and $m_2 = ((1,4),(2,3))$.  When $T$ is a classical $2$-tangle with standard orientation, the local move $T \mapsto \rho^2(T)$ corresponds to the traditional notation of mutation.  This means that following proposition is standard orientation version of the classic result stating that the auxiliary polynomial is invariant under mutation.

\begin{proposition}
\label{thm: 2-tangles, mutation}
Let $T$ be a classical $2$-tangle with standard orientation.  For any $2$-tangle $T'$ with compatible orientation, $f(\rho^2(T) \cup T') = f(T \cup T')$.
\end{proposition}
\begin{proof}
Observe that $r^2(m_1) = m_1$ and $r^2(m_2) = m_2$, giving $f(T(r^2(m_1))) - f(T(m_1)) = 0$ and $f(T(r^2(m_2))) - f(T(m_2)) = 0$.  Proposition \ref{thm: divisibility under rotational move} then implies that $f(\rho^2(T) \cup T') - f(T \cup T')$ is divisible by $0$ for any $T'$.
\end{proof}

Our next local move represents a ``semi-mutation" on the associated diagram.  For a demonstration of how the closures of a $2$-tangle behave under this move, see Figure \ref{fig: semi-mutation example}.

\begin{theorem}
\label{thm: 2-tangles, semi-mutation}
Let $T$ be a classical $2$-tangle with standard orientation.  For any $2$-tangle $T'$ with compatible orientation, $f(\rho^1(T) \cup T') - f(T \cup T')$ is divisible by $f(T(m_1)) - f(T(m_2))$.
\end{theorem}
\begin{proof}
Here we have $r^1(m_1) = m_2$ and $r^1(m_2) = m_1$.  Thus $f(T(r^1(m_1))) - f(T(m_1)) = f(T(m_2)) - f(T(m_1))$ and $f(T(r^1(m_2))) - f(T(m_2)) = f(T(m_1)) - f(T(m_2))$, from which Proposition \ref{thm: divisibility under rotational move} gives the desired result.
\end{proof}

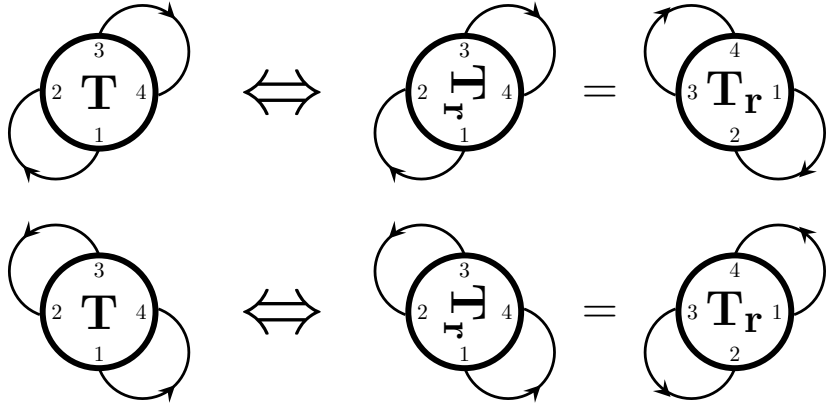
\begin{figure}[ht!]
\centering

\scalebox{.75}{
\begin{tikzpicture}
[scale=1,auto=left,every node/.style={circle,inner sep=0pt}]
\draw[line width=3pt] (0,0) circle (1cm);
\node (T) at (0:0) {\Huge{\textbf{T}}};
\node (1) at (-90:1) {};
\node (2) at (-180:1) {};
\node (3) at (-270:1) {};
\node (4) at (-360:1) {};
\node (1*) at (-90:.75) {\small{1}};
\node (2*) at (-180:.75) {\small{2}};
\node (3*) at (-270:.75) {\small{3}};
\node (4*) at (-360:.75) {\small{4}};
\draw[line width=1.4pt,->-] (1) circle arc (-20:-250:.82);
\draw[line width=1.4pt,->-] (3) circle arc (-200:-430:.82);
\end{tikzpicture}
\hspace{.2in}
\raisebox{37pt}{\scalebox{3.5}{$\Leftrightarrow$}}
\hspace{.2in}
\begin{tikzpicture}
[scale=1,auto=left,every node/.style={circle,inner sep=0pt}]
\draw[line width=3pt] (0,0) circle (1cm);
\node[rotate=-90] (T) at (0:0) {\Huge{\textbf{T\textsubscript{r}}}};
\node (1) at (-90:1) {};
\node (2) at (-180:1) {};
\node (3) at (-270:1) {};
\node (4) at (-360:1) {};
\node (1*) at (-90:.75) {\small{1}};
\node (2*) at (-180:.75) {\small{2}};
\node (3*) at (-270:.75) {\small{3}};
\node (4*) at (-360:.75) {\small{4}};
\draw[line width=1.4pt,->-] (1) circle arc (-20:-250:.82);
\draw[line width=1.4pt,->-] (3) circle arc (-200:-430:.82);
\end{tikzpicture}
\hspace{.03in}
\raisebox{42pt}{\scalebox{2}{$=$}}
\hspace{.03in}
\begin{tikzpicture}
[scale=1,auto=left,every node/.style={circle,inner sep=0pt}]
\draw[line width=3pt] (0,0) circle (1cm);
\node (T) at (0:0) {\Huge{\textbf{T\textsubscript{r}}}};
\node (1) at (-90:1) {};
\node (2) at (-180:1) {};
\node (3) at (-270:1) {};
\node (4) at (-360:1) {};
\node (1*) at (-90:.75) {\small{2}};
\node (2*) at (-180:.75) {\small{3}};
\node (3*) at (-270:.75) {\small{4}};
\node (4*) at (-360:.75) {\small{1}};
\draw[line width=1.4pt,-<-] (1) circle arc (-160:70:.82);
\draw[line width=1.4pt,-<-] (3) circle arc (-340:-110:.82);
\end{tikzpicture}}

\vspace{.15in}

\scalebox{.75}{
\begin{tikzpicture}
[scale=1,auto=left,every node/.style={circle,inner sep=0pt}]
\draw[line width=3pt] (0,0) circle (1cm);
\node (T) at (0:0) {\Huge{\textbf{T}}};
\node (1) at (-90:1) {};
\node (2) at (-180:1) {};
\node (3) at (-270:1) {};
\node (4) at (-360:1) {};
\node (1*) at (-90:.75) {\small{1}};
\node (2*) at (-180:.75) {\small{2}};
\node (3*) at (-270:.75) {\small{3}};
\node (4*) at (-360:.75) {\small{4}};
\draw[line width=1.4pt,->-] (1) circle arc (-160:70:.82);
\draw[line width=1.4pt,->-] (3) circle arc (-340:-110:.82);
\end{tikzpicture}
\hspace{.2in}
\raisebox{37pt}{\scalebox{3.5}{$\Leftrightarrow$}}
\hspace{.2in}
\begin{tikzpicture}
[scale=1,auto=left,every node/.style={circle,inner sep=0pt}]
\draw[line width=3pt] (0,0) circle (1cm);
\node[rotate=-90] (T) at (0:0) {\Huge{\textbf{T\textsubscript{r}}}};
\node (1) at (-90:1) {};
\node (2) at (-180:1) {};
\node (3) at (-270:1) {};
\node (4) at (-360:1) {};
\node (1*) at (-90:.75) {\small{1}};
\node (2*) at (-180:.75) {\small{2}};
\node (3*) at (-270:.75) {\small{3}};
\node (4*) at (-360:.75) {\small{4}};
\draw[line width=1.4pt,->-] (1) circle arc (-160:70:.82);
\draw[line width=1.4pt,->-] (3) circle arc (-340:-110:.82);
\end{tikzpicture}
\hspace{.03in}
\raisebox{42pt}{\scalebox{2}{$=$}}
\hspace{.03in}
\begin{tikzpicture}
[scale=1,auto=left,every node/.style={circle,inner sep=0pt}]
\draw[line width=3pt] (0,0) circle (1cm);
\node (T) at (0:0) {\Huge{\textbf{T\textsubscript{r}}}};
\node (1) at (-90:1) {};
\node (2) at (-180:1) {};
\node (3) at (-270:1) {};
\node (4) at (-360:1) {};
\node (1*) at (-90:.75) {\small{2}};
\node (2*) at (-180:.75) {\small{3}};
\node (3*) at (-270:.75) {\small{4}};
\node (4*) at (-360:.75) {\small{1}};
\draw[line width=1.4pt,-<-] (1) circle arc (-20:-250:.82);
\draw[line width=1.4pt,-<-] (3) circle arc (-200:-430:.82);
\end{tikzpicture}}

\caption{Both closures $T(m_i) \mapsto \rho^1(T)(m_i)$ for a $2$-tangle $T$ undergoing the local move of Theorem \ref{thm: 2-tangles, semi-mutation}.  Here $T_r$ denotes that the orientation of every strand in $T$ has been reversed.}
\label{fig: semi-mutation example}
\end{figure}

\begin{example}
\label{ex: semi-mutation example 1}
Let $T$ contain an even number $k$ of positive half-twists, giving our local move $T \mapsto \rho^1(T)$ of the form

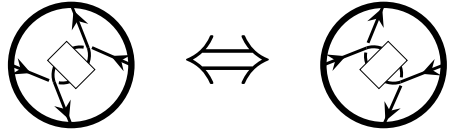
\begin{center}
\scalebox{.8}{
\begin{tikzpicture}
[scale=1,auto=left,every node/.style={circle,inner sep=0pt}]
\draw[line width=3pt] (0,0) circle (1cm);
\node (1) at (-90:1) {};
\node (2) at (-180:1) {};
\node (3) at (-270:1) {};
\node (4) at (-360:1) {};
\draw[line width=1.4pt,-<-] (1) to (-135:.4);
\draw[line width=1.4pt,->-] (2) to (-145:.45);
\draw[line width=1.4pt,-<-] (3) to (45:.4);
\draw[line width=1.4pt,->-] (4) to (40:.45);
\draw[line width=1.4pt,bend left=60] (-135:.4) to (55:.36);
\draw[line width=1.4pt,bend right=60] (-125:.36) to (45:.4);
\draw[rectangle,fill=white,rotate=-45] (146:.4) rectangle (-34:.4);
\end{tikzpicture}
\hspace{.2in}
\raisebox{22pt}{\scalebox{3.5}{$\Leftrightarrow$}}
\hspace{.2in}
\begin{tikzpicture}
[scale=1,rotate=-90,auto=left,every node/.style={circle,inner sep=0pt}]
\draw[line width=3pt] (0,0) circle (1cm);
\node (1) at (-90:1) {};
\node (2) at (-180:1) {};
\node (3) at (-270:1) {};
\node (4) at (-360:1) {};
\draw[line width=1.4pt,->-] (1) to (-135:.4);
\draw[line width=1.4pt,-<-] (2) to (-145:.45);
\draw[line width=1.4pt,->-] (3) to (45:.4);
\draw[line width=1.4pt,-<-] (4) to (40:.45);
\draw[line width=1.4pt,bend left=60] (-135:.4) to (55:.36);
\draw[line width=1.4pt,bend right=60] (-125:.36) to (45:.4);
\draw[rectangle,fill=white,rotate=-45] (146:.4) rectangle (-34:.4);
\end{tikzpicture}}
\end{center}

\noindent Here $T(m_1)$ is the unknot, while $T(m_2)$ is the $(2,k)$-torus link $L_{(2,k)}$.  Theorem \ref{thm: 2-tangles, semi-mutation} then states $f(\rho^1(T) \cup T') - f(T \cup T')$ is divisible by $1-f(L_{(2,k)})$ for any $2$-tangle $T'$, where $f(L_{(2,k)}) = \dfrac{(-1)^{k+1}A^{-2(k-1)}(1-A^{-12}+(-1)^k(A^{-4-4k}-A^{-8-4k}))}{1-A^{-8}}$.  By construction, this divisor $1 - f(L_{(2,k)})$ is maximal for all classical links with respect to $T \mapsto \rho^1(T)$.

\end{example}

\begin{example}
\label{ex: semi-mutation example 2}
Let $T$ contain any number $k$ of positive half-twists followed by a clasp, giving $T \mapsto \rho^1(T)$ of the form 

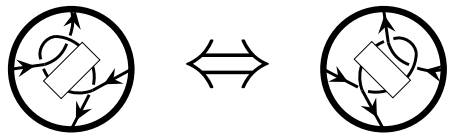
\begin{center}
\scalebox{.8}{
\begin{tikzpicture}
[scale=1,auto=left,every node/.style={circle,inner sep=0pt}]
\draw[line width=3pt] (0,0) circle (1cm);
\node (1) at (-90:1) {};
\node (2) at (-180:1) {};
\node (3) at (-270:1) {};
\node (4) at (-360:1) {};
\draw[line width=1.4pt,-->-] (1) to (-55:.52);
\draw[line width=1.4pt,--<-] (4) to (-45:.5);
\draw[line width=1.4pt,--<-] (2) to (170:.45);
\draw[line width=1.4pt,-->-,bend left=15] (3) to (92:.56);
\draw[line width=1.4pt,bend right=25] (170:.45) to (100:.43);
\draw[line width=1.4pt,bend left=45] (-45:.5) to (179:.46);
\draw[line width=1.4pt,bend right=50] (-35:.45) to (115:.6);
\draw[line width=1.4pt,bend left=55] (162:.53) to (115:.6);
\draw[rectangle,fill=white,rotate=45] (158:.5) rectangle (-27:.5);
\end{tikzpicture}
\hspace{.2in}
\raisebox{22pt}{\scalebox{3.5}{$\Leftrightarrow$}}
\hspace{.2in}
\begin{tikzpicture}
[scale=1,rotate=-90,auto=left,every node/.style={circle,inner sep=0pt}]
\draw[line width=3pt] (0,0) circle (1cm);
\node (1) at (-90:1) {};
\node (2) at (-180:1) {};
\node (3) at (-270:1) {};
\node (4) at (-360:1) {};
\draw[line width=1.4pt,--<-] (1) to (-55:.52);
\draw[line width=1.4pt,-->-] (4) to (-45:.5);
\draw[line width=1.4pt,-->-] (2) to (170:.45);
\draw[line width=1.4pt,--<-,bend left=15] (3) to (92:.56);
\draw[line width=1.4pt,bend right=25] (170:.45) to (100:.43);
\draw[line width=1.4pt,bend left=45] (-45:.5) to (179:.46);
\draw[line width=1.4pt,bend right=50] (-35:.45) to (115:.6);
\draw[line width=1.4pt,bend left=55] (162:.53) to (115:.6);
\draw[rectangle,fill=white,rotate=45] (158:.5) rectangle (-27:.5);
\end{tikzpicture}}
\end{center}

\noindent Here $T(m_1) = \mathcal{T}_k$ is the twist knot with $k$ half-twists, and $T(m_2) = H$ is the Hopf link.  It follows that $f(\rho^1(T) \cup T') - f(T \cup T')$ is divisible by $f(\mathcal{T}_k)-f(H)$ for any $T'$, where $f(H) = -A^{-4} - A^4$ and

\begin{center}
    $f(\mathcal{T}_k) = \begin{cases} \dfrac{1+A^{-4}+(-1)^kA^{-2k}-(-1)^{k+1}A^{-2k-6}}{-A^2+1} & \;\ k \;\ \text{odd}  \\[2.5ex]
\dfrac{-A^{6}-A^{2}-(-1)^kA^{6-2k}+(-1)^kA^{-2k}}{-A^2+1} & \;\ k \;\ \text{even}  
\end{cases}$
\end{center}

\noindent By construction, this divisor $f(\mathcal{T}_k) - f(H)$ is maximal for all classical links with respect to this particular local move $T \mapsto \rho^1(T)$.
\end{example}

We now expand our attention to rotational local moves on classical $n$-tangles for $n > 2$.  In the case of classical $3$-tangles, we may exploit of the symmetry of matchings in $\mathcal{M}_3$ to derive a specialization of Proposition \ref{thm: divisibility under rotational move} whose divisibility conditions are especially simple.

Denote the five elements of $\mathcal{M}_3$ by $m_{a_1} = ((1,2),(3,4),(5,6))$, $m_{a_2} = ((1,6),(2,3),(4,5))$, $m_{b_1} = ((1,2),(3,6),(4,5))$, $m_{b_2}=((1,4),(2,3),(5,6))$, and $m_{b_3} = ((1,6),(2,5),(3,4))$.  We then consider the local move $T \mapsto \rho^3(T)$, a $\pi$-radian rotation whose effect on these various closures is shown in Figure \ref{fig: 3-tangles under mutation}.

\begin{theorem}
\label{thm: 3-tangles, mutation}
Let $T$ be a classical $3$-tangle with standard orientation.  For any $3$-tangle $T'$ with compatible orientation, $f(\rho^3(T) \cup T') - f(T \cup T')$ is divisible by $f(T(m_{a_1})) - f(T(m_{a_2}))$.
\end{theorem}
\begin{proof}
We have $r^3(m_{a_1}) = m_{a_2}$, $r^3(m_{a_2}) = m_{a_1}$, and $r^3(m_{b_i}) = m_{b_i}$ for $i=1,2,3$.  It follows that $f(T(r^3(m_{a_1}))) - f(T(m_{a_1})) = f(T(m_{a_2})) - f(T(m_{a_1})) = - f(T(r^3(m_{a_2}))) + f(T(m_{a_2}))$ and $f(T(r^3(m_{b_i}))) - f(T(m_{b_i})) = 0$ for $i=1,2,3$.  The theorem then follows from Proposition \ref{thm: divisibility under rotational move}.
\end{proof}

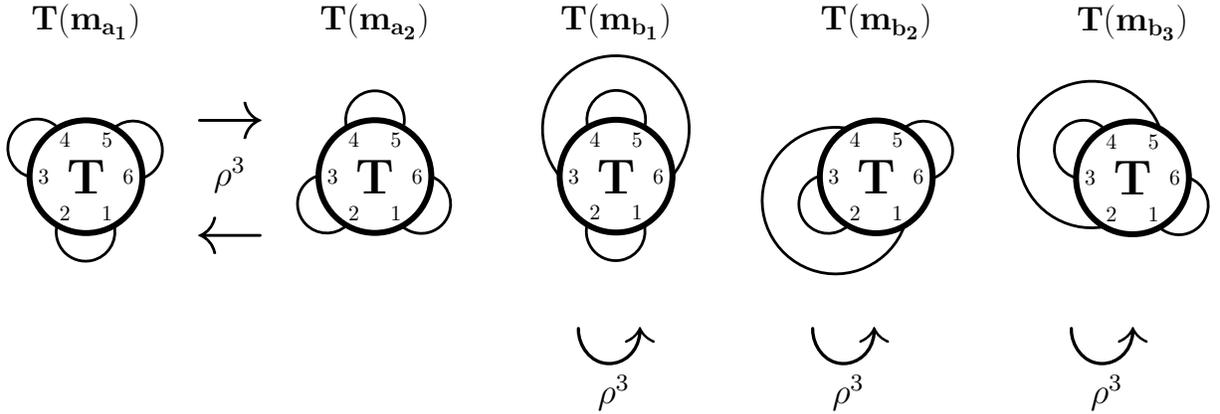
\begin{figure}[ht!]
\centering
\setlength{\tabcolsep}{2pt}
\begin{tabular}{c c c}
\scalebox{.75}{
\begin{tikzpicture}
[scale=1,auto=left,every node/.style={circle,inner sep=0pt}]
\draw[line width=3pt] (0,0) circle (1cm);
\node (T) at (0:0) {\Huge{\textbf{T}}};
\node (1) at (-60:1) {};
\node (2) at (-120:1) {};
\node (3) at (-180:1) {};
\node (4) at (-240:1) {};
\node (5) at (-300:1) {};
\node (6) at (-360:1) {};
\node (1*) at (-60:.75) {\small{1}};
\node (2*) at (-120:.75) {\small{2}};
\node (3*) at (-180:.75) {\small{3}};
\node (4*) at (-240:.75) {\small{4}};
\node (5*) at (-300:.75) {\small{5}};
\node (6*) at (-360:.75) {\small{6}};
\draw[line width=1.4pt] (3) circle arc (-105:-310:.52);
\draw[line width=1.4pt] (5) circle arc (130:-75:.52);
\draw[line width=1.4pt] (1) circle arc (12.5:-192.5:.52);
\node (Ta1) at (90:2.75) {\Large{$\mathbf{T(m_{a_1})}$}};
\end{tikzpicture}} &
\raisebox{28pt}{
\scalebox{.75}{
$\begin{matrix}
\scalebox{3}{$\rightarrow$} \\
\raisebox{5pt}{\scalebox{1.5}{$\rho^3$}} \\
\scalebox{3}{$\leftarrow$}
\end{matrix}$}} &
\raisebox{10pt}{
\scalebox{.75}{
\begin{tikzpicture}
[scale=1,auto=left,every node/.style={circle,inner sep=0pt}]
\draw[line width=3pt] (0,0) circle (1cm);
\node (T) at (0:0) {\Huge{\textbf{T}}};
\node (1) at (-60:1) {};
\node (2) at (-120:1) {};
\node (3) at (-180:1) {};
\node (4) at (-240:1) {};
\node (5) at (-300:1) {};
\node (6) at (-360:1) {};
\node (1*) at (-60:.75) {\small{1}};
\node (2*) at (-120:.75) {\small{2}};
\node (3*) at (-180:.75) {\small{3}};
\node (4*) at (-240:.75) {\small{4}};
\node (5*) at (-300:.75) {\small{5}};
\node (6*) at (-360:.75) {\small{6}};
\draw[line width=1.4pt] (4) circle arc (-165:-370:.52);
\draw[line width=1.4pt] (6) circle arc (70:-135:.52);
\draw[line width=1.4pt] (2) circle arc (-47.5:-252.5:.52);
\node (Ta2) at (90:2.75) {\Large{$\mathbf{T(m_{a_2})}$}};
\end{tikzpicture}}}
\end{tabular}
\hspace{.03in}
\raisebox{-25pt}{
\setlength{\tabcolsep}{7pt}
\begin{tabular}{c c c}
\scalebox{.75}{
\begin{tikzpicture}
[scale=1,auto=left,every node/.style={circle,inner sep=0pt}]
\draw[line width=3pt] (0,0) circle (1cm);
\node (T) at (0:0) {\Huge{\textbf{T}}};
\node (1) at (-60:1) {};
\node (2) at (-120:1) {};
\node (3) at (-180:1) {};
\node (4) at (-240:1) {};
\node (5) at (-300:1) {};
\node (6) at (-360:1) {};
\node (1*) at (-60:.75) {\small{1}};
\node (2*) at (-120:.75) {\small{2}};
\node (3*) at (-180:.75) {\small{3}};
\node (4*) at (-240:.75) {\small{4}};
\node (5*) at (-300:.75) {\small{5}};
\node (6*) at (-360:.75) {\small{6}};
\draw[line width=1.4pt] (1) circle arc (12.5:-192.5:.52);
\draw[line width=1.4pt] (4) circle arc (-165:-370:.52);
\draw[line width=1.4pt] (3) circle arc (220:-40:1.3);
\node (Tb1) at (90:2.75) {\Large{$\mathbf{T(m_{b_1})}$}};
\end{tikzpicture}} &
\raisebox{-8pt}{
\scalebox{.75}{
\begin{tikzpicture}
[scale=1,auto=left,every node/.style={circle,inner sep=0pt}]
\draw[line width=3pt] (0,0) circle (1cm);
\node (T) at (0:0) {\Huge{\textbf{T}}};
\node (1) at (-60:1) {};
\node (2) at (-120:1) {};
\node (3) at (-180:1) {};
\node (4) at (-240:1) {};
\node (5) at (-300:1) {};
\node (6) at (-360:1) {};
\node (1*) at (-60:.75) {\small{1}};
\node (2*) at (-120:.75) {\small{2}};
\node (3*) at (-180:.75) {\small{3}};
\node (4*) at (-240:.75) {\small{4}};
\node (5*) at (-300:.75) {\small{5}};
\node (6*) at (-360:.75) {\small{6}};
\draw[line width=1.4pt] (5) circle arc (130:-75:.52);
\draw[line width=1.4pt] (2) circle arc (-47.5:-252.5:.52);
\draw[line width=1.4pt] (1) circle arc (340:80:1.3);
\node (Tb2) at (90:2.75) {\Large{$\mathbf{T(m_{b_2})}$}};
\end{tikzpicture}}} &
\raisebox{10pt}{
\scalebox{.75}{
\begin{tikzpicture}
[scale=1,auto=left,every node/.style={circle,inner sep=0pt}]
\draw[line width=3pt] (0,0) circle (1cm);
\node (T) at (0:0) {\Huge{\textbf{T}}};
\node (1) at (-60:1) {};
\node (2) at (-120:1) {};
\node (3) at (-180:1) {};
\node (4) at (-240:1) {};
\node (5) at (-300:1) {};
\node (6) at (-360:1) {};
\node (1*) at (-60:.75) {\small{1}};
\node (2*) at (-120:.75) {\small{2}};
\node (3*) at (-180:.75) {\small{3}};
\node (4*) at (-240:.75) {\small{4}};
\node (5*) at (-300:.75) {\small{5}};
\node (6*) at (-360:.75) {\small{6}};
\draw[line width=1.4pt] (3) circle arc (-105:-310:.52);
\draw[line width=1.4pt] (6) circle arc (70:-135:.52);
\draw[line width=1.4pt] (2) circle arc (280:20:1.3);
\node (Tb3) at (90:2.75) {\Large{$\mathbf{T(m_{b_3})}$}};
\end{tikzpicture}}} \\
\scalebox{.75}{
$\begin{matrix}
\scalebox{3.5}{\rotatebox{180}{$\curvearrowleft$}} \\
\scalebox{1.5}{$\rho^3$} 
\end{matrix}$} &
\scalebox{.75}{
$\begin{matrix}
\scalebox{3.5}{\rotatebox{180}{$\curvearrowleft$}} \\
\scalebox{1.5}{$\rho^3$} 
\end{matrix}$} &
\scalebox{.75}{
$\begin{matrix}
\scalebox{3.5}{\rotatebox{180}{$\curvearrowleft$}} \\
\scalebox{1.5}{$\rho^3$} 
\end{matrix}$}
\end{tabular}}
\caption{Up to reversal of all strands, the way in which the local move $T(m_i) \mapsto \rho^3(T)(m_i)$ from Theorem \ref{thm: 3-tangles, mutation} permutes the five closures of the classical $3$-tangle $T$.}
\label{fig: 3-tangles under mutation}
\end{figure}

Perhaps the best known example of a local move involving $3$-tangles is the $\Delta$-move, an unknotting operation whose unorientated form is shown in Figure \ref{fig: delta move}.  As shown by Murakami and Nakanishi \cite{MurakamiNakanishi}, the $\Delta$-move possesses the interesting property that two links $L_1,L_2$ may be connected by a single $\Delta$-move (of any orientation) if and only if $L_1,L_2$ may be connected by a single oriented $\Delta$-move with standard orientation.  Since a $\Delta$-move with standard orientation qualifies as a local move of the form $T \mapsto \rho^3(T)$, Theorem \ref{thm: 3-tangles, mutation} may then be applied to give a significantly simpler proof of the $\Delta$-move divisibility theorem presented by Ganzell \cite{Ganzell}: 

\begin{theorem}
\label{thm: delta move}
Let $L$ and $L'$ be a pair of classical links that are separated by a single $\Delta$-move.  Then $f(L) - f(L')$ is divisible by $A^{16} - A^{12} - A^4 + 1$.  Furthermore, this divisor is maximal for all classical links with respect to the $\Delta$-move.
\end{theorem}
\begin{proof}
For a standard orientation version of the tangle on the left side of Figure \ref{fig: delta move}, $T(m_{a_1})$ is the unknot and $T(m_{a_2})$ is the left-handed trefoil.  In the case of standard orientation, Theorem \ref{thm: 3-tangles, mutation} then states that $f(L)-f(L')$ is then divisible by $f(1) - f(3_1) = A^{16} - A^{12} - A^4 + 1$.  The general result follows from the observation of Murakami and Nakanishi \cite{MurakamiNakanishi}.  The fact that this divisor is maximal follows from the fact that it equals $f(1) - f(3_1)$.
\end{proof}

\begin{figure}[ht!]
\centering
\scalebox{.8}{
\begin{tikzpicture}
[scale=1,auto=left,every node/.style={circle,inner sep=0pt}]
\draw[line width=3pt] (0,0) circle (1cm);
\node (1) at (-60:1) {};
\node (2) at (-120:1) {};
\node (3) at (-180:1) {};
\node (4) at (-240:1) {};
\node (5) at (-300:1) {};
\node (6) at (-360:1) {};
\draw[line width=1.4pt] (1) to (-80:.58);
\draw[line width=1.4pt,bend right=25] (4) to (-95:.46);
\draw[line width=1.4pt,bend right=25] (2) to (-335:.46);
\draw[line width=1.4pt] (-320:.58) to (5);
\draw[line width=1.4pt,bend right=25] (6) to (-215:.46);
\draw[line width=1.4pt] (-200:.58) to (3);
\end{tikzpicture}
\hspace{.2in}
\raisebox{22pt}{\scalebox{3.5}{$\Leftrightarrow$}}
\hspace{.2in}
\begin{tikzpicture}
[scale=1,rotate=180,auto=left,every node/.style={circle,inner sep=0pt}]
\draw[line width=3pt] (0,0) circle (1cm);
\node (1) at (-60:1) {};
\node (2) at (-120:1) {};
\node (3) at (-180:1) {};
\node (4) at (-240:1) {};
\node (5) at (-300:1) {};
\node (6) at (-360:1) {};
\draw[line width=1.4pt] (1) to (-80:.58);
\draw[line width=1.4pt,bend right=25] (4) to (-95:.46);
\draw[line width=1.4pt,bend right=25] (2) to (-335:.46);
\draw[line width=1.4pt] (-320:.58) to (5);
\draw[line width=1.4pt,bend right=25] (6) to (-215:.46);
\draw[line width=1.4pt] (-200:.58) to (3);
\end{tikzpicture}}
\caption{The $\Delta$ move.}
\label{fig: delta move}
\end{figure}
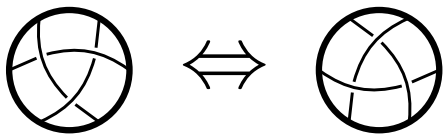

The method of Theorem \ref{thm: 3-tangles, mutation} may be adapted to local moves $T \mapsto \rho^2(T)$ where the classical $3$-tangle is rotated by $\frac{2 \pi}{3}$-radians.  However, due to the manner in which a $\frac{2 \pi}{3}$-radian rotation permutes the various closures of $T$, the result is somewhat less elegant:

\begin{theorem}
\label{thm: 3-tangles, one-third rotation}
Let $T$ be a classical $3$-tangle with standard orientation.  For any $3$-tangle $T'$ with compatible orientation, $f(\rho^2(T) \cup T') - f(T \cup T')$ is divisible by the greatest common divisor of $f(T(m_{b_2})) - f(T(m_{b_1}))$ and $f(T(m_{b_3})) - f(T(m_{b_2}))$.
\end{theorem}
\begin{proof}
For this rotation $r^2(m_{a_1}) = m_{a_1}$, $r^2(m_{a_2}) = m_{a_2}$, $r^2(m_{b_1}) = b_2$, $r^2(m_{b_2}) = m_{b_3}$, and $r^2(m_{b_3}) = m_{b_1}$.  By Proposition \ref{thm: divisibility under rotational move}, it follows that $f(\rho^2(T) \cup T') - f(T \cup T')$ is divisible by the greatest common divisor of $f(T(m_{b_2})) - f(T(m_{b_1}))$, $f(T(m_{b_3})) - f(T(m_{b_2}))$, and $f(T(m_{b_1})) - f(T(m_{b_3}))$.  The result follows from the fact that any divisor of those first two differences is necessarily a divisor of the third difference.
\end{proof}


\subsection{The double-$\Delta$-move}
\label{subsec: double-delta}

For the remainder of this paper, we explore a handful of additional local moves to which the results of Subsection \ref{subsec: rotational local moves} do not immediately apply.

We begin with the double-$\Delta$-move for classical $6$-tangles, whose unoriented form is shown in Figure \ref{fig: double-delta move}.  As originally defined by Naik and Stanford \cite{NaikStanford}, the double-$\Delta$-move is assumed to involve tangles where each pair of parallel strands are oriented in opposite directions.

\begin{figure}[ht!]
\centering
\scalebox{.9}{
\begin{tikzpicture}
[scale=1,auto=left,every node/.style={circle,inner sep=0pt}]
\node (1) at (-60:1) {};
\node (2) at (-120:1) {};
\node (3) at (-180:1) {};
\node (4) at (-240:1) {};
\node (5) at (-300:1) {};
\node (6) at (-360:1) {};
\draw[line width=6.3pt] (1) to (-80:.58);
\draw[line width=6.3pt] (-320:.58) to (5);
\draw[line width=6.3pt] (-200:.58) to (3);
\draw[line width=3.5pt,color=white] (1) to (-80:.58);
\draw[line width=3.5pt,color=white] (-320:.58) to (5);
\draw[line width=3.5pt,color=white] (-200:.58) to (3);
\draw[line width=9.5pt,bend right=25,color=white] (4) to (-100:.41);
\draw[line width=9.5pt,bend right=25,color=white] (2) to (-340:.41);
\draw[line width=9.5pt,bend right=25,color=white] (6) to (-220:.41);
\draw[line width=6.3pt,bend right=25] (4) to (-100:.41);
\draw[line width=6.3pt,bend right=25] (2) to (-340:.41);
\draw[line width=6.3pt,bend right=25] (6) to (-220:.41);
\draw[line width=3.5pt,bend right=25,color=white] (4) to (-99:.42);
\draw[line width=3.5pt,bend right=25,color=white] (2) to (-339:.42);
\draw[line width=3.5pt,bend right=25,color=white] (6) to (-219:.42);
\draw[line width=3.5pt] (0,0) circle (1cm);
\node (1l) at (-52:1.3) {\small{1}};
\node (2l) at (-68:1.3) {\small{2}};
\node (3l) at (-112:1.3) {\small{3}};
\node (4l) at (-130:1.3) {\small{4}};
\node (5l) at (-171:1.25) {\small{5}};
\node (6l) at (-189:1.25) {\small{6}};
\node (7l) at (-232:1.3) {\small{7}};
\node (8l) at (-250:1.3) {\small{8}};
\node (9l) at (-291:1.3) {\small{9}};
\node (10l) at (-308:1.3) {\small{10}};
\node (11l) at (-352:1.3) {\small{11}};
\node (12l) at (-370:1.3) {\small{12}};
\end{tikzpicture}
\hspace{.2in}
\raisebox{26pt}{\scalebox{3.5}{$\Leftrightarrow$}}
\hspace{.2in}
\begin{tikzpicture}
[scale=1,auto=left,every node/.style={circle,inner sep=0pt},rotate=180]
\node (1) at (-60:1) {};
\node (2) at (-120:1) {};
\node (3) at (-180:1) {};
\node (4) at (-240:1) {};
\node (5) at (-300:1) {};
\node (6) at (-360:1) {};
\draw[line width=6.3pt] (1) to (-80:.58);
\draw[line width=6.3pt] (-320:.58) to (5);
\draw[line width=6.3pt] (-200:.58) to (3);
\draw[line width=3.5pt,color=white] (1) to (-80:.58);
\draw[line width=3.5pt,color=white] (-320:.58) to (5);
\draw[line width=3.5pt,color=white] (-200:.58) to (3);
\draw[line width=9.5pt,bend right=25,color=white] (4) to (-100:.41);
\draw[line width=9.5pt,bend right=25,color=white] (2) to (-340:.41);
\draw[line width=9.5pt,bend right=25,color=white] (6) to (-220:.41);
\draw[line width=6.3pt,bend right=25] (4) to (-100:.41);
\draw[line width=6.3pt,bend right=25] (2) to (-340:.41);
\draw[line width=6.3pt,bend right=25] (6) to (-220:.41);
\draw[line width=3.5pt,bend right=25,color=white] (4) to (-99:.42);
\draw[line width=3.5pt,bend right=25,color=white] (2) to (-339:.42);
\draw[line width=3.5pt,bend right=25,color=white] (6) to (-219:.42);
\draw[line width=3.5pt] (0,0) circle (1cm);
\node (1l) at (-52:1.3) {\small{7}};
\node (2l) at (-68:1.3) {\small{8}};
\node (3l) at (-112:1.3) {\small{9}};
\node (4l) at (-130:1.3) {\small{10}};
\node (5l) at (-171:1.3) {\small{11}};
\node (6l) at (-189:1.3) {\small{12}};
\node (7l) at (-233:1.3) {\small{1}};
\node (8l) at (-249:1.3) {\small{2}};
\node (9l) at (-291:1.3) {\small{3}};
\node (10l) at (-308:1.3) {\small{4}};
\node (11l) at (-352:1.25) {\small{5}};
\node (12l) at (-370:1.25) {\small{6}};
\end{tikzpicture}}
\caption{The double-$\Delta$-move.}
\label{fig: double-delta move}
\end{figure}
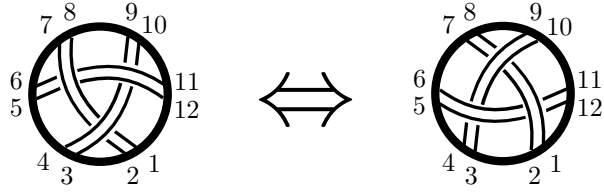

It is straightforward to show that two links $L_1,L_2$ may be connected by a finite sequence of double-$\Delta$-moves if and only if $L_1,L_2$ may be connected by a finite sequences of double-$\Delta$-moves that all involve standard orientation $6$-tangles.  As shown in Figure \ref{fig: double-delta orientation bypass}, any non-standard orientation double-$\Delta$-move may be bypassed by performing a Reidemeister II move on any pair of strands that aren't in the proper orientation, and then redefining the tangle boundary to obtain a standard orientation $6$-tangle.  As was the case with the original $\Delta$-move, this technique allows us to prove a general divisibility result for the double-$\Delta$-move merely by checking divisibility of $f(L_2)-f(L_1)$ in the case of standard standard orientation.

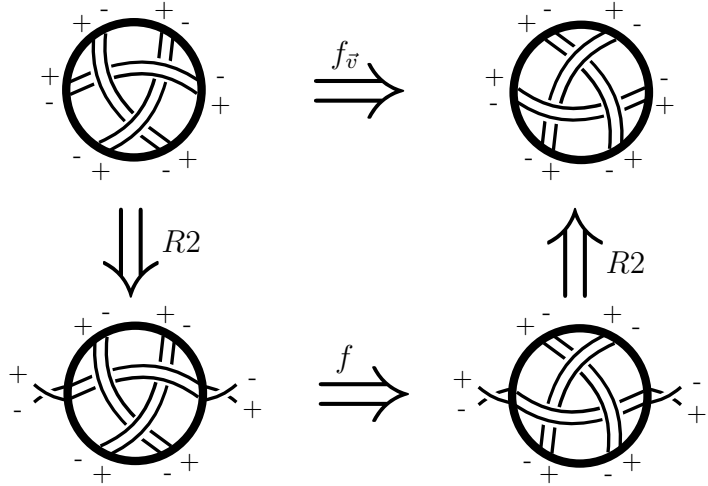
\begin{figure}[ht!]
\centering
\setlength{\tabcolsep}{8pt}
\begin{tabular}{c c c}
\scalebox{.9}{
\begin{tikzpicture}
[scale=1,auto=left,every node/.style={circle,inner sep=0pt}]
\node (1) at (-60:1) {};
\node (2) at (-120:1) {};
\node (3) at (-180:1) {};
\node (4) at (-240:1) {};
\node (5) at (-300:1) {};
\node (6) at (-360:1) {};
\draw[line width=6.3pt] (1) to (-80:.58);
\draw[line width=6.3pt] (-320:.58) to (5);
\draw[line width=6.3pt] (-200:.58) to (3);
\draw[line width=3.5pt,color=white] (1) to (-80:.58);
\draw[line width=3.5pt,color=white] (-320:.58) to (5);
\draw[line width=3.5pt,color=white] (-200:.58) to (3);
\draw[line width=9.5pt,bend right=25,color=white] (4) to (-100:.41);
\draw[line width=9.5pt,bend right=25,color=white] (2) to (-340:.41);
\draw[line width=9.5pt,bend right=25,color=white] (6) to (-220:.41);
\draw[line width=6.3pt,bend right=25] (4) to (-100:.41);
\draw[line width=6.3pt,bend right=25] (2) to (-340:.41);
\draw[line width=6.3pt,bend right=25] (6) to (-220:.41);
\draw[line width=3.5pt,bend right=25,color=white] (4) to (-99:.42);
\draw[line width=3.5pt,bend right=25,color=white] (2) to (-339:.42);
\draw[line width=3.5pt,bend right=25,color=white] (6) to (-219:.42);
\draw[line width=3.5pt] (0,0) circle (1cm);
\node (1l) at (-52:1.3) {\small{+}};
\node (2l) at (-68:1.3) {\small{-}};
\node (3l) at (-112:1.3) {\small{+}};
\node (4l) at (-130:1.3) {\small{-}};
\node (5l) at (-171:1.25) {\small{-}};
\node (6l) at (-189:1.25) {\small{+}};
\node (7l) at (-232:1.25) {\small{+}};
\node (8l) at (-250:1.25) {\small{-}};
\node (9l) at (-291:1.25) {\small{+}};
\node (10l) at (-308:1.25) {\small{-}};
\node (11l) at (-352:1.3) {\small{-}};
\node (12l) at (-370:1.3) {\small{+}};
\end{tikzpicture}}
&
\raisebox{26pt}{\scalebox{3.2}{$\Rightarrow$} \kern-34pt \raisebox{20pt}{$f_{\vec{v}}$}}
\hspace{.15in}
&
\scalebox{.9}{
\begin{tikzpicture}
[scale=1,auto=left,every node/.style={circle,inner sep=0pt},rotate=180]
\node (1) at (-60:1) {};
\node (2) at (-120:1) {};
\node (3) at (-180:1) {};
\node (4) at (-240:1) {};
\node (5) at (-300:1) {};
\node (6) at (-360:1) {};
\draw[line width=6.3pt] (1) to (-80:.58);
\draw[line width=6.3pt] (-320:.58) to (5);
\draw[line width=6.3pt] (-200:.58) to (3);
\draw[line width=3.5pt,color=white] (1) to (-80:.58);
\draw[line width=3.5pt,color=white] (-320:.58) to (5);
\draw[line width=3.5pt,color=white] (-200:.58) to (3);
\draw[line width=9.5pt,bend right=25,color=white] (4) to (-100:.41);
\draw[line width=9.5pt,bend right=25,color=white] (2) to (-340:.41);
\draw[line width=9.5pt,bend right=25,color=white] (6) to (-220:.41);
\draw[line width=6.3pt,bend right=25] (4) to (-100:.41);
\draw[line width=6.3pt,bend right=25] (2) to (-340:.41);
\draw[line width=6.3pt,bend right=25] (6) to (-220:.41);
\draw[line width=3.5pt,bend right=25,color=white] (4) to (-99:.42);
\draw[line width=3.5pt,bend right=25,color=white] (2) to (-339:.42);
\draw[line width=3.5pt,bend right=25,color=white] (6) to (-219:.42);
\draw[line width=3.5pt] (0,0) circle (1cm);
\node (1l) at (-52:1.3) {\small{+}};
\node (2l) at (-68:1.3) {\small{-}};
\node (3l) at (-112:1.3) {\small{+}};
\node (4l) at (-130:1.3) {\small{-}};
\node (5l) at (-171:1.25) {\small{-}};
\node (6l) at (-189:1.25) {\small{+}};
\node (7l) at (-232:1.25) {\small{+}};
\node (8l) at (-250:1.25) {\small{-}};
\node (9l) at (-291:1.25) {\small{+}};
\node (10l) at (-308:1.25) {\small{-}};
\node (11l) at (-352:1.3) {\small{-}};
\node (12l) at (-370:1.3) {\small{+}};
\end{tikzpicture}}
\\
\hspace{.15in} \scalebox{3.2}{$\Downarrow$} \kern-4pt \raisebox{10pt}{$R2$}
&
&
\hspace{.15in} \scalebox{3.2}{$\Uparrow$} \kern-4pt \raisebox{2pt}{$R2$}
\\
\scalebox{.9}{
\begin{tikzpicture}
[scale=1,auto=left,every node/.style={circle,inner sep=0pt}]
\node (1) at (-60:1) {};
\node (2) at (-120:1) {};
\node (3) at (-180:1) {};
\node (4) at (-240:1) {};
\node (5) at (-300:1) {};
\node (6) at (-360:1) {};
\draw[line width=6.3pt] (1) to (-80:.58);
\draw[line width=6.3pt] (-320:.58) to (5);
\draw[line width=6.3pt] (-200:.58) to (3);
\draw[line width=3.5pt,color=white] (1) to (-80:.58);
\draw[line width=3.5pt,color=white] (-320:.58) to (5);
\draw[line width=3.5pt,color=white] (-200:.58) to (3);
\draw[line width=9.5pt,bend right=25,color=white] (4) to (-100:.41);
\draw[line width=9.5pt,bend right=25,color=white] (2) to (-340:.41);
\draw[line width=9.5pt,bend right=25,color=white] (6) to (-220:.41);
\draw[line width=6.3pt,bend right=25] (4) to (-100:.41);
\draw[line width=6.3pt,bend right=25] (2) to (-340:.41);
\draw[line width=6.3pt,bend right=25] (6) to (-220:.41);
\draw[line width=3.5pt,bend right=25,color=white] (4) to (-99:.42);
\draw[line width=3.5pt,bend right=25,color=white] (2) to (-339:.42);
\draw[line width=3.5pt,bend right=25,color=white] (6) to (-219:.42);
\draw[line width=1.4pt,bend right=15] (-186:1) to (-175:1.5);
\draw[line width=4.4pt,bend left=15,color=white] (-174:1) to (-185:1.5);
\draw[line width=1.4pt,bend left=15] (-174:1) to (-185:1.5);
\draw[line width=1.4pt,bend left=15] (6:1) to (-5:1.5);
\draw[line width=4.4pt,bend right=15,color=white] (-6:1) to (5:1.5);
\draw[line width=1.4pt,bend right=15] (-6:1) to (5:1.5);
\draw[line width=3.5pt] (0,0) circle (1cm);
\node (1l) at (-52:1.3) {\small{+}};
\node (2l) at (-68:1.3) {\small{-}};
\node (3l) at (-112:1.3) {\small{+}};
\node (4l) at (-130:1.3) {\small{-}};
\node (5l) at (-171:1.75) {\small{-}};
\node (6l) at (-189:1.75) {\small{+}};
\node (7l) at (-232:1.25) {\small{+}};
\node (8l) at (-250:1.25) {\small{-}};
\node (9l) at (-291:1.25) {\small{+}};
\node (10l) at (-308:1.25) {\small{-}};
\node (11l) at (-353:1.75) {\small{-}};
\node (12l) at (-368:1.75) {\small{+}};
\end{tikzpicture}}
&
\raisebox{26pt}{\scalebox{3.2}{$\Rightarrow$} \kern-34pt \raisebox{20pt}{$f$}} \hspace{.15in}
&
\scalebox{.9}{
\begin{tikzpicture}
[scale=1,auto=left,every node/.style={circle,inner sep=0pt},rotate=180]
\node (1) at (-60:1) {};
\node (2) at (-120:1) {};
\node (3) at (-180:1) {};
\node (4) at (-240:1) {};
\node (5) at (-300:1) {};
\node (6) at (-360:1) {};
\draw[line width=6.3pt] (1) to (-80:.58);
\draw[line width=6.3pt] (-320:.58) to (5);
\draw[line width=6.3pt] (-200:.58) to (3);
\draw[line width=3.5pt,color=white] (1) to (-80:.58);
\draw[line width=3.5pt,color=white] (-320:.58) to (5);
\draw[line width=3.5pt,color=white] (-200:.58) to (3);
\draw[line width=9.5pt,bend right=25,color=white] (4) to (-100:.41);
\draw[line width=9.5pt,bend right=25,color=white] (2) to (-340:.41);
\draw[line width=9.5pt,bend right=25,color=white] (6) to (-220:.41);
\draw[line width=6.3pt,bend right=25] (4) to (-100:.41);
\draw[line width=6.3pt,bend right=25] (2) to (-340:.41);
\draw[line width=6.3pt,bend right=25] (6) to (-220:.41);
\draw[line width=3.5pt,bend right=25,color=white] (4) to (-99:.42);
\draw[line width=3.5pt,bend right=25,color=white] (2) to (-339:.42);
\draw[line width=3.5pt,bend right=25,color=white] (6) to (-219:.42);
\draw[line width=1.4pt,bend left=15] (-174:1) to (-185:1.5);
\draw[line width=4.4pt,bend right=15,color=white] (-186:1) to (-175:1.5);
\draw[line width=1.4pt,bend right=15] (-186:1) to (-175:1.5);
\draw[line width=1.4pt,bend right=15] (-6:1) to (5:1.5);
\draw[line width=4.4pt,bend left=15,color=white] (6:1) to (-5:1.5);
\draw[line width=1.4pt,bend left=15] (6:1) to (-5:1.5);
\draw[line width=3.5pt] (0,0) circle (1cm);
\node (1l) at (-52:1.3) {\small{+}};
\node (2l) at (-68:1.3) {\small{-}};
\node (3l) at (-112:1.3) {\small{+}};
\node (4l) at (-130:1.3) {\small{-}};
\node (5l) at (-174:1.75) {\small{-}};
\node (6l) at (-189:1.75) {\small{+}};
\node (7l) at (-232:1.25) {\small{+}};
\node (8l) at (-250:1.25) {\small{-}};
\node (9l) at (-291:1.25) {\small{+}};
\node (10l) at (-308:1.25) {\small{-}};
\node (11l) at (-353:1.75) {\small{-}};
\node (12l) at (-368:1.75) {\small{+}};
\end{tikzpicture}}
\end{tabular}
\caption{Replacing a non-standard orientation double-$\Delta$-move $f_{\vec{v}}$ with a standard-orientation double-$\Delta$-move $f$ in a manner that does not change the underlying links.}
\label{fig: double-delta orientation bypass}
\end{figure}

\begin{theorem}
\label{thm: double-delta move}
Let $L$ and $L'$ be a pair of classical links that are separated by a single double-$\Delta$-move.  Then $f(L) - f(L')$ is divisible by $A^{36}-A^{32}+A^{28}-A^{24}-A^{12}+A^8-A^4+1$.  Furthermore, this divisor is maximal for all classical links with respect to the double-$\Delta$-move.
\end{theorem}
\begin{proof}
As the double-$\Delta$-move involves classical tangles with standard orientation, Proposition \ref{thm: general auxiliary polynomial divisibility} requires that we determine $f(T_2(m))-f(T_1(m))$ for all $C_6 = 132$ elements of $\mathcal{M}_6$.  However, using the endpoint numbering shown in Figure \ref{fig: double-delta move} , it is clear that $T_1(m) = T_2(m)$ for any $m \in \mathcal{M}_6$ that includes an arc of the form $(i,i+1)$ for at least one odd integer $i$.  This leaves fifteen closures for which $f(T_2(m))-f(T_1(m))$ may be nonzero.

Now observe that both tangles $T_1,T_2$ involved in the double-$\Delta$-move are invariant under rotation by four strands.  It follows that $f(T_2(m))-f(T_1(m)) = f(T_2(m'))-f(T_1(m'))$ for any pair of closures $m,m' \in \mathcal{M}_6$ that differ via rotation by four strands.  Among our fifteen remaining closures, this reduces the necessary computations to the seven closures below.

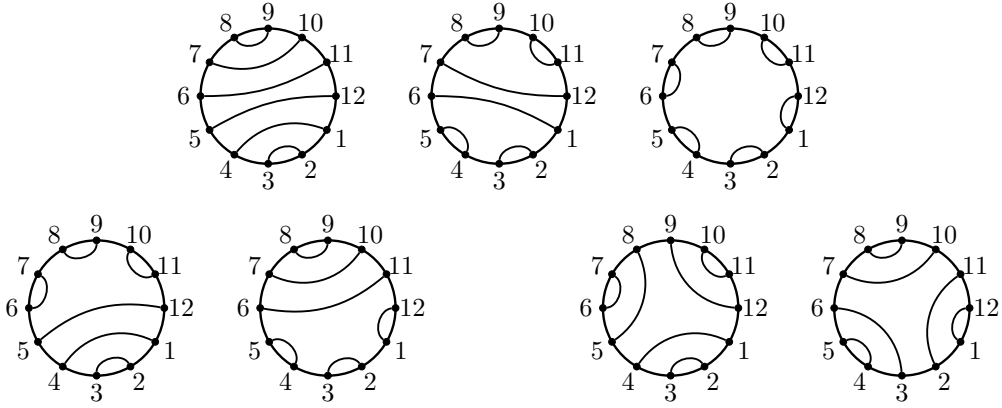
\begin{center}
\scalebox{.9}{
\begin{tikzpicture}
[scale=1,auto=left,every node/.style={circle,inner sep=0pt}]
	\draw[line width=1pt] (0,0) circle (1cm);
	\node[draw,fill,inner sep=1pt] (1) at (-30:1) {};
	\node[draw,fill,inner sep=1pt] (2) at (-60:1) {};
	\node[draw,fill,inner sep=1pt] (3) at (-90:1) {};
	\node[draw,fill,inner sep=1pt] (4) at (-120:1) {};
	\node[draw,fill,inner sep=1pt] (5) at (-150:1) {};
	\node[draw,fill,inner sep=1pt] (6) at (-180:1) {};
	\node[draw,fill,inner sep=1pt] (7) at (-210:1) {};
	\node[draw,fill,inner sep=1pt] (8) at (-240:1) {};
	\node[draw,fill,inner sep=1pt] (9) at (-270:1) {};
	\node[draw,fill,inner sep=1pt] (10) at (-300:1) {};
	\node[draw,fill,inner sep=1pt] (11) at (-330:1) {};
	\node[draw,fill,inner sep=1pt] (12) at (-360:1) {};
	\node (1l) at (-30:1.25) {\small{1}};
	\node (2l) at (-60:1.25) {\small{2}};
	\node (3l) at (-90:1.25) {\small{3}};
	\node (4l) at (-120:1.25) {\small{4}};
	\node (5l) at (-150:1.25) {\small{5}};
	\node (6l) at (-180:1.25) {\small{6}};
	\node (7l) at (-210:1.25) {\small{7}};
	\node (8l) at (-240:1.25) {\small{8}};
	\node (9l) at (-270:1.25) {\small{9}};
	\node (10l) at (-300:1.25) {\small{10}};
	\node (11l) at (-330:1.25) {\small{11}};
	\node (12l) at (-360:1.25) {\small{12}};	
	\draw[thick, bend right=35] (1) to (4);
	\draw[thick, bend right=70] (2) to (3);
	\draw[thick, bend left=15] (5) to (12);
	\draw[thick, bend right=15] (6) to (11);
	\draw[thick, bend right=35] (7) to (10);
	\draw[thick, bend right=70] (8) to (9);	
\end{tikzpicture}
\hspace{.1in}
\begin{tikzpicture}
[scale=1,auto=left,every node/.style={circle,inner sep=0pt}]
	\draw[line width=1pt] (0,0) circle (1cm);
	\node[draw,fill,inner sep=1pt] (1) at (-30:1) {};
	\node[draw,fill,inner sep=1pt] (2) at (-60:1) {};
	\node[draw,fill,inner sep=1pt] (3) at (-90:1) {};
	\node[draw,fill,inner sep=1pt] (4) at (-120:1) {};
	\node[draw,fill,inner sep=1pt] (5) at (-150:1) {};
	\node[draw,fill,inner sep=1pt] (6) at (-180:1) {};
	\node[draw,fill,inner sep=1pt] (7) at (-210:1) {};
	\node[draw,fill,inner sep=1pt] (8) at (-240:1) {};
	\node[draw,fill,inner sep=1pt] (9) at (-270:1) {};
	\node[draw,fill,inner sep=1pt] (10) at (-300:1) {};
	\node[draw,fill,inner sep=1pt] (11) at (-330:1) {};
	\node[draw,fill,inner sep=1pt] (12) at (-360:1) {};
	\node (1l) at (-30:1.25) {\small{1}};
	\node (2l) at (-60:1.25) {\small{2}};
	\node (3l) at (-90:1.25) {\small{3}};
	\node (4l) at (-120:1.25) {\small{4}};
	\node (5l) at (-150:1.25) {\small{5}};
	\node (6l) at (-180:1.25) {\small{6}};
	\node (7l) at (-210:1.25) {\small{7}};
	\node (8l) at (-240:1.25) {\small{8}};
	\node (9l) at (-270:1.25) {\small{9}};
	\node (10l) at (-300:1.25) {\small{10}};
	\node (11l) at (-330:1.25) {\small{11}};
	\node (12l) at (-360:1.25) {\small{12}};	
	\draw[thick, bend right=15] (1) to (6);
	\draw[thick, bend right=15] (7) to (12);
	\draw[thick, bend right=70] (2) to (3);
	\draw[thick, bend right=70] (4) to (5);
	\draw[thick, bend right=70] (8) to (9);
	\draw[thick, bend right=70] (10) to (11);	
\end{tikzpicture}
\hspace{.1in}
\begin{tikzpicture}
[scale=1,auto=left,every node/.style={circle,inner sep=0pt}]
	\draw[line width=1pt] (0,0) circle (1cm);
	\node[draw,fill,inner sep=1pt] (1) at (-30:1) {};
	\node[draw,fill,inner sep=1pt] (2) at (-60:1) {};
	\node[draw,fill,inner sep=1pt] (3) at (-90:1) {};
	\node[draw,fill,inner sep=1pt] (4) at (-120:1) {};
	\node[draw,fill,inner sep=1pt] (5) at (-150:1) {};
	\node[draw,fill,inner sep=1pt] (6) at (-180:1) {};
	\node[draw,fill,inner sep=1pt] (7) at (-210:1) {};
	\node[draw,fill,inner sep=1pt] (8) at (-240:1) {};
	\node[draw,fill,inner sep=1pt] (9) at (-270:1) {};
	\node[draw,fill,inner sep=1pt] (10) at (-300:1) {};
	\node[draw,fill,inner sep=1pt] (11) at (-330:1) {};
	\node[draw,fill,inner sep=1pt] (12) at (-360:1) {};
	\node (1l) at (-30:1.25) {\small{1}};
	\node (2l) at (-60:1.25) {\small{2}};
	\node (3l) at (-90:1.25) {\small{3}};
	\node (4l) at (-120:1.25) {\small{4}};
	\node (5l) at (-150:1.25) {\small{5}};
	\node (6l) at (-180:1.25) {\small{6}};
	\node (7l) at (-210:1.25) {\small{7}};
	\node (8l) at (-240:1.25) {\small{8}};
	\node (9l) at (-270:1.25) {\small{9}};
	\node (10l) at (-300:1.25) {\small{10}};
	\node (11l) at (-330:1.25) {\small{11}};
	\node (12l) at (-360:1.25) {\small{12}};	
	\draw[thick, bend left=70] (1) to (12);
	\draw[thick, bend right=70] (2) to (3);
	\draw[thick, bend right=70] (4) to (5);
	\draw[thick, bend right=70] (6) to (7);
	\draw[thick, bend right=70] (8) to (9);
	\draw[thick, bend right=70] (10) to (11);	
\end{tikzpicture}}

\vspace{.1in}

\scalebox{.9}{
\begin{tikzpicture}
[scale=1,auto=left,every node/.style={circle,inner sep=0pt}]
	\draw[line width=1pt] (0,0) circle (1cm);
	\node[draw,fill,inner sep=1pt] (1) at (-30:1) {};
	\node[draw,fill,inner sep=1pt] (2) at (-60:1) {};
	\node[draw,fill,inner sep=1pt] (3) at (-90:1) {};
	\node[draw,fill,inner sep=1pt] (4) at (-120:1) {};
	\node[draw,fill,inner sep=1pt] (5) at (-150:1) {};
	\node[draw,fill,inner sep=1pt] (6) at (-180:1) {};
	\node[draw,fill,inner sep=1pt] (7) at (-210:1) {};
	\node[draw,fill,inner sep=1pt] (8) at (-240:1) {};
	\node[draw,fill,inner sep=1pt] (9) at (-270:1) {};
	\node[draw,fill,inner sep=1pt] (10) at (-300:1) {};
	\node[draw,fill,inner sep=1pt] (11) at (-330:1) {};
	\node[draw,fill,inner sep=1pt] (12) at (-360:1) {};
	\node (1l) at (-30:1.25) {\small{1}};
	\node (2l) at (-60:1.25) {\small{2}};
	\node (3l) at (-90:1.25) {\small{3}};
	\node (4l) at (-120:1.25) {\small{4}};
	\node (5l) at (-150:1.25) {\small{5}};
	\node (6l) at (-180:1.25) {\small{6}};
	\node (7l) at (-210:1.25) {\small{7}};
	\node (8l) at (-240:1.25) {\small{8}};
	\node (9l) at (-270:1.25) {\small{9}};
	\node (10l) at (-300:1.25) {\small{10}};
	\node (11l) at (-330:1.25) {\small{11}};
	\node (12l) at (-360:1.25) {\small{12}};	
	\draw[thick, bend right=40] (1) to (4);
	\draw[thick, bend right=75] (2) to (3);
	\draw[thick, bend left=25] (5) to (12);
	\draw[thick, bend right=75] (6) to (7);
	\draw[thick, bend right=75] (8) to (9);
	\draw[thick, bend right=75] (10) to (11);
\end{tikzpicture}
\hspace{.1in}
\begin{tikzpicture}
[scale=1,auto=left,every node/.style={circle,inner sep=0pt}]
	\draw[line width=1pt] (0,0) circle (1cm);
	\node[draw,fill,inner sep=1pt] (1) at (-30:1) {};
	\node[draw,fill,inner sep=1pt] (2) at (-60:1) {};
	\node[draw,fill,inner sep=1pt] (3) at (-90:1) {};
	\node[draw,fill,inner sep=1pt] (4) at (-120:1) {};
	\node[draw,fill,inner sep=1pt] (5) at (-150:1) {};
	\node[draw,fill,inner sep=1pt] (6) at (-180:1) {};
	\node[draw,fill,inner sep=1pt] (7) at (-210:1) {};
	\node[draw,fill,inner sep=1pt] (8) at (-240:1) {};
	\node[draw,fill,inner sep=1pt] (9) at (-270:1) {};
	\node[draw,fill,inner sep=1pt] (10) at (-300:1) {};
	\node[draw,fill,inner sep=1pt] (11) at (-330:1) {};
	\node[draw,fill,inner sep=1pt] (12) at (-360:1) {};
	\node (1l) at (-30:1.25) {\small{1}};
	\node (2l) at (-60:1.25) {\small{2}};
	\node (3l) at (-90:1.25) {\small{3}};
	\node (4l) at (-120:1.25) {\small{4}};
	\node (5l) at (-150:1.25) {\small{5}};
	\node (6l) at (-180:1.25) {\small{6}};
	\node (7l) at (-210:1.25) {\small{7}};
	\node (8l) at (-240:1.25) {\small{8}};
	\node (9l) at (-270:1.25) {\small{9}};
	\node (10l) at (-300:1.25) {\small{10}};
	\node (11l) at (-330:1.25) {\small{11}};
	\node (12l) at (-360:1.25) {\small{12}};	
	\draw[thick, bend left=70] (1) to (12);
	\draw[thick, bend right=70] (2) to (3);
	\draw[thick, bend right=70] (4) to (5);
	\draw[thick, bend right=25] (6) to (11);
	\draw[thick, bend right=40] (7) to (10);
	\draw[thick, bend right=70] (8) to (9);	
\end{tikzpicture}
\hspace{.75in}
\begin{tikzpicture}
[scale=1,auto=left,every node/.style={circle,inner sep=0pt}]
	\draw[line width=1pt] (0,0) circle (1cm);
	\node[draw,fill,inner sep=1pt] (1) at (-30:1) {};
	\node[draw,fill,inner sep=1pt] (2) at (-60:1) {};
	\node[draw,fill,inner sep=1pt] (3) at (-90:1) {};
	\node[draw,fill,inner sep=1pt] (4) at (-120:1) {};
	\node[draw,fill,inner sep=1pt] (5) at (-150:1) {};
	\node[draw,fill,inner sep=1pt] (6) at (-180:1) {};
	\node[draw,fill,inner sep=1pt] (7) at (-210:1) {};
	\node[draw,fill,inner sep=1pt] (8) at (-240:1) {};
	\node[draw,fill,inner sep=1pt] (9) at (-270:1) {};
	\node[draw,fill,inner sep=1pt] (10) at (-300:1) {};
	\node[draw,fill,inner sep=1pt] (11) at (-330:1) {};
	\node[draw,fill,inner sep=1pt] (12) at (-360:1) {};
	\node (1l) at (-30:1.25) {\small{1}};
	\node (2l) at (-60:1.25) {\small{2}};
	\node (3l) at (-90:1.25) {\small{3}};
	\node (4l) at (-120:1.25) {\small{4}};
	\node (5l) at (-150:1.25) {\small{5}};
	\node (6l) at (-180:1.25) {\small{6}};
	\node (7l) at (-210:1.25) {\small{7}};
	\node (8l) at (-240:1.25) {\small{8}};
	\node (9l) at (-270:1.25) {\small{9}};
	\node (10l) at (-300:1.25) {\small{10}};
	\node (11l) at (-330:1.25) {\small{11}};
	\node (12l) at (-360:1.25) {\small{12}};	
	\draw[thick, bend right=40] (1) to (4);
	\draw[thick, bend right=70] (2) to (3);
	\draw[thick, bend right=40] (5) to (8);
	\draw[thick, bend right=70] (6) to (7);
	\draw[thick, bend right=40] (9) to (12);
	\draw[thick, bend right=70] (10) to (11);
\end{tikzpicture}
\hspace{.1in}
\begin{tikzpicture}
[scale=1,auto=left,every node/.style={circle,inner sep=0pt}]
	\draw[line width=1pt] (0,0) circle (1cm);
	\node[draw,fill,inner sep=1pt] (1) at (-30:1) {};
	\node[draw,fill,inner sep=1pt] (2) at (-60:1) {};
	\node[draw,fill,inner sep=1pt] (3) at (-90:1) {};
	\node[draw,fill,inner sep=1pt] (4) at (-120:1) {};
	\node[draw,fill,inner sep=1pt] (5) at (-150:1) {};
	\node[draw,fill,inner sep=1pt] (6) at (-180:1) {};
	\node[draw,fill,inner sep=1pt] (7) at (-210:1) {};
	\node[draw,fill,inner sep=1pt] (8) at (-240:1) {};
	\node[draw,fill,inner sep=1pt] (9) at (-270:1) {};
	\node[draw,fill,inner sep=1pt] (10) at (-300:1) {};
	\node[draw,fill,inner sep=1pt] (11) at (-330:1) {};
	\node[draw,fill,inner sep=1pt] (12) at (-360:1) {};
	\node (1l) at (-30:1.25) {\small{1}};
	\node (2l) at (-60:1.25) {\small{2}};
	\node (3l) at (-90:1.25) {\small{3}};
	\node (4l) at (-120:1.25) {\small{4}};
	\node (5l) at (-150:1.25) {\small{5}};
	\node (6l) at (-180:1.25) {\small{6}};
	\node (7l) at (-210:1.25) {\small{7}};
	\node (8l) at (-240:1.25) {\small{8}};
	\node (9l) at (-270:1.25) {\small{9}};
	\node (10l) at (-300:1.25) {\small{10}};
	\node (11l) at (-330:1.25) {\small{11}};
	\node (12l) at (-360:1.25) {\small{12}};	
	\draw[thick, bend right=40] (3) to (6);
	\draw[thick, bend right=70] (4) to (5);
	\draw[thick, bend right=40] (7) to (10);
	\draw[thick, bend right=70] (8) to (9);
	\draw[thick, bend right=40] (11) to (2);
	\draw[thick, bend right=70] (12) to (1);
\end{tikzpicture}}
\end{center}

Among these seven remaining closures, it may be shown that we still have $T_1(m) = T_2(m)$ for all three closures in the top row.  For the first two closures in the second row, $f(T_1(m))$ and $f(T_2(m))$ are $A^{24} + A^{16} + A^8 + 1$ and $A^{20} + A^{12} + A^8 + 2 - A^{-4} + A^{-8} - A^{-12}$ (in some order).  For the last two closures in the second row, $T_1(m)$ and $T_2(m)$ are a 2-cable of the unknot and a $2$-cable of the trefoil (in some order), giving $f(T_1(m))$ and $f(T_2(m))$ of $-A^{26} - A^{18} + A^{14} - A^{10} + A^6 - A^2$ and $-A^{18} - A^{10} - A^2 + A^{-10} + A^{-18} - A^{-22}$.  Taking the greatest common divisor of these non-trivial differences gives the desired result.

To see that our divisor is maximal among all classical links, notice that $A^{36} - A^{32} + A^{28} - A^{24} - A^{12} + A^8 - A^4 + 1 = f(1) - f(K)$, where $K = 11_{42}$ is the Kinoshita-Terasaka knot.  As the Kinoshita-Terasaka has the same Alexander polynomial as the unknot, it follows from the work of Naik and Stanford \cite{NaikStanford} that it may be transformed into the unknot via a finite sequence of double-$\Delta$-moves.
\end{proof}

One significant application of Theorem \ref{thm: double-delta move} involves S-equivalence of knots.  A pair of classical knots $K,K'$ are said to be S-equivalent if their Seifert matrices are related by a sequence of elementary enlargements and similarity.  Knots in the same S-equivalence class share many interesting properties, such as having identical Alexander polynomials and isometric Blanchfield pairings.  More pertinent to this paper is the work of Naik and Stanford \cite{NaikStanford}, who showed that two oriented classical knots are S-equivalent if and only if they are related by a sequence of double-$\Delta$-moves.  This fact immediately prompts the following corollary of Theorem \ref{thm: double-delta move}:

\begin{corollary}
\label{thm: S-equivalence divisibility}
Let $K,K'$ be classical knots such that $K$ and $K'$ are S-equivalent.  Then $f(K)-f(K')$ is divisible by $A^{36}-A^{32}+A^{28}-A^{24}-A^{12}+A^8-A^4+1$.
\end{corollary}

\subsection{The Virtual Rotation Move}
\label{subsec: virtual rotation}

Most of the results in this section do not easily generalize to virtual tangles.  This derives from the fact that the virtual move of Figure \ref{fig: virtual crossing replacement}, which was necessary for our derivation of Theorem \ref{thm: auxiliary polynomial decomposition}, gives all virtual crossings in $T(m) - T$ a ``preferred" quadrant that is fixed as $T$ undergoes the local move.  In particular, for the rotational move $T \mapsto \rho^k(T)$, the modified closures of Theorem \ref{thm: auxiliary polynomial decomposition} do not obey $\widetilde{\rho^k(T)}\kern-0pt ^B(m) = \widetilde{T}^B(r^k(m))$ if $m$ contains at least one virtual crossing.

Luckily, some local moves $T_1 \mapsto T_2$ involving virtual tangles are simple enough that it is possible to ignore Theorem \ref{thm: auxiliary polynomial decomposition} and manually calculate a similar, better-suited decomposition for $f(T_1 \cup T')$ and $f(T_2 \cup T')$.  One such move is a generalization of the ``semi-mutation" move from Theorem \ref{thm: 2-tangles, semi-mutation} to virtual $2$-tangles.

\begin{theorem}
\label{thm: virtual 2-tangle rotation}

Let $T$ be a virtual $2$-tangle with standard orientation.  For any $2$-tangle $T'$ with compatible orientation, $f(T \cup T') - f(\rho^1(T) \cup T')$ is divisible by $f(T(m_1)) - f(T(m_2))$.

\end{theorem}
\begin{proof}
Denoting the closures as in Figure \ref{fig: 4-strand closures}, Proposition \ref{thm: bracket polynomial decomposition} immediately gives the following decompositions, where $q_i \in \Z[A,A^{-1}]$:

$$\langle T \cup T' \rangle = q_1 \langle T(m_1) \rangle + q_2 \langle T(m_2) \rangle + q_3 \langle T(m_3) \rangle$$
$$\langle \rho^1(T) \cup T' \rangle = q_1 \langle \rho^1(T(m_1)) \rangle + q_2 \langle \rho^1(T(m_2)) \rangle + q_3 \langle \rho^1(T(m_3)) \rangle$$
$$= q_1 \langle T(m_2) \rangle + q_2 \langle T(m_1) \rangle + q_3 \langle T(m_3)) \rangle$$

Translating from the Kauffman bracket to the auxiliary polynomial requires that we replace $T(m_3)$ with closures that respect the standard orientation.  The Kauffman skein relation gives $\langle T(m_3) \rangle = A^{-1} \langle T(m_v) \rangle - A^{-2} \langle T(m_2) \rangle$, with $T(m_v)$ as shown in the upper-right corner of Figure \ref{fig: n=2 tangle decompositions}.  Absorbing the various writhe terms $(-A^3)^{-w}$ from each $f(T(m_i))$ into the the leading Laurent polynomials gives the following decompositions, where $p_i \in \Z[A,A^{-1}]$ and $T_r$ denotes the tangle obtained by reversing every strand in $T$.

$$f(T \cup T') = p_1 f(T(m_1)) + p_2 f(T(m_2)) + \left[ p_3 f(T(m_v)) - p_4 f(T(m_2)) \right]$$
$$f(\rho^1 (T) \cup T') = p_1 f(T_r(m_2)) + p_2 f(T_r(m_1)) + \left[ p_3 f(T_r(m_v)) + p_4 f(T_r(m_2)) \right]$$

Noting that $f(T_r(m))=f(T(m))$ for any closure $m$, we conclude that $f(\rho^1 (T) \cup T') - f(T \cup T')$ must be divisible by $f(T_1(m)) - f(T(m_2))$.
\end{proof}

\end{document}